\documentclass[a4paper,11pt,fleqn]{article}

\usepackage{amsmath,amssymb,amsthm,graphicx}%,enumerate,cite,backref
\usepackage[font=footnotesize,labelsep=period,justification=justified]{caption}
\usepackage[title]{appendix}

\usepackage{hyperref}
\hypersetup{colorlinks, linkcolor=blue, citecolor=blue, urlcolor=blue}

\allowdisplaybreaks

\newcommand{\p}{\partial}
\newcommand{\ord}{\mathop{\rm ord}\nolimits}
\newcommand{\lcoef}{\mathop{\rm lcoef}\nolimits}

\newcommand{\sgn}{\mathop{\rm sgn}\nolimits}
\newcommand{\pr}{\mathop{\rm pr}\nolimits}
\newcommand{\supp}{\mathop{\rm supp}\nolimits}
\newcommand{\DO}{{\rm DO}}

\newcommand{\lara}[2]{\langle #1, #2 \rangle}
\newcommand{\blara}[2]{\big\langle #1, #2 \big\rangle}

\newcommand{\todo}[1][\null]{\ensuremath{\clubsuit}}

\newcommand{\noprint}[1]{}
\newcommand{\checked}[1][\null]{\ensuremath{\boldsymbol{\surd}}}

\newtheorem{theorem}{Theorem}[section]
\newtheorem{lemma}[theorem]{Lemma}
\newtheorem{corollary}[theorem]{Corollary}
\newtheorem{proposition}[theorem]{Proposition}

{\theoremstyle{definition}
\newtheorem{definition}[theorem]{Definition}
\newtheorem{example}[theorem]{Example}
\newtheorem{remark}[theorem]{Remark}
\newtheorem*{notation*}{Notation}
}

\begin{document}
%\par\noindent {\LARGE\bf
\par\noindent {\LARGE\bf
Parameter-dependent linear ordinary differential\\ equations and topology of domains
\par}

\vspace{4mm}\par\noindent{\large
Vyacheslav M.~Boyko$^{\dag 1}$, Michael Kunzinger$^{\ddag 2}$ and Roman O.~Popovych$^{\dag\ddag\S 3}$
\par\vspace{2mm}\par}

\vspace{2mm}\par\noindent{\it
${}^\dag$\,Institute of Mathematics of National Academy of Sciences of Ukraine, \\
$\phantom{{}^\dag}$\,3 Tereshchenkivska Str., Kyiv-4, 01601 Ukraine\\[2mm]
$^\ddag$\,Fakult\"at f\"ur Mathematik, Universit\"at Wien, Oskar-Morgenstern-Platz 1, 1090 Wien, Austria\\[2mm]
$^\S$\,Mathematical Institute, Silesian University in Opava, Na Rybn\'\i{}\v{c}ku 1, 746 01 Opava, \\
$\phantom{^\S}$\,Czech Republic}\\[2mm]
E-mail:
$^1$boyko@imath.kiev.ua,
$^2$michael.kunzinger@univie.ac.at,
$^3$rop@imath.kiev.ua
\par
{\vspace{8mm}\par\noindent\hspace*{8mm}\parbox{140mm}{\small
The well-known solution theory for (systems of) linear ordinary differential equations undergoes
significant changes when introducing an additional real parameter. Properties like the existence of
fundamental sets of solutions or characterizations of such sets via nonvanishing
Wronskians are sensitive to the topological properties of the underlying domain
of the independent variable and the parameter. We give a complete
characterization of the solvability of such parameter-dependent equations and systems in terms
of topological properties of the domain. In addition, we also investigate this
problem in the setting of Schwartz distributions.
\vskip 1em

\medskip

}\par\vspace{6mm}}

\section{Introduction}\label{sec:intro}

The solution theory of $p$th order linear ordinary differential equations (ODEs)
\begin{equation}\label{eq:ode_no_parameter}
\sum_{i=0}^p g^i(x)\frac{{\rm d}^iu}{{\rm d} x^i}=f(x),
\end{equation}
where $u$ is the unknown function,
the independent variable~$x$ varies in some interval $(a,b)$, $g^0,\dots,g^p,f\in {\rm C}\big((a,b)\big)$ and $g^p(x)\not=0$ for all $x\in (a,b)$,
is a classical subject that is a part of most textbooks in the field (e.g., \cite{Amann1990,Hartman2002,Walter1998}).
The set of all (classical) solutions to the homogeneous equation ($f=0$) forms a $p$-dimensional vector subspace~$V$ of
${\rm C}^p\big((a,b)\big)$, while the solution space of \eqref{eq:ode_no_parameter} is an affine subspace obtained by translating~$V$
by any particular solution of \eqref{eq:ode_no_parameter}. Any tuple of $p$ linearly independent solutions $(\varphi^1,\dots,\varphi^p)$
of the homogeneous equation is called a fundamental set of solutions, and setting $\varphi_{s-1}:={\rm d}^{s-1}\varphi/{\rm d} x^{s-1}$, $s=1,\dots,p$, we write
\begin{equation}\label{eq:def_wronskian}
\mathrm W(\varphi^1,\dots,\varphi^p) := \det
\begin{pmatrix}
\varphi^1 & \dots & \varphi^p \\
\varphi^1_1 & \dots & \varphi^p_1 \\
\vdots & \ddots & \vdots \\
\varphi^1_{p-1} & \dots & \varphi^p_{p-1}
\end{pmatrix}
\end{equation}
for the corresponding Wronskian.
Solutions $\varphi^1,\dots,\varphi^p$ of the homogeneous equation form a fundamental set
if and only if the Wronskian $\mathrm W(\varphi^1,\dots,\varphi^p)$ does not vanish at $x_0\in (a,b)$,
and therefore it vanishes nowhere on~$(a,b)$ in view of the Liouville--Ostrogradski formula,
\[
\mathrm W(\varphi^1,\dots,\varphi^p)(x)=\mathrm W(\varphi^1,\dots,\varphi^p)(x_0)\exp\left(-\int_{x_0}^x\frac{g^{p-1}(x')}{g^p(x')}\,{\rm d}x'\right).
\]
In this case, $(\varphi^1,\dots,\varphi^p)$
is a basis of the vector space of solutions to the homogeneous equation.
See, e.g., \cite[Section~IV.8.iii]{Hartman2002} or \cite[Section~19.II]{Walter1998}.
A particular solution to the inhomogeneous equation
\eqref{eq:ode_no_parameter} is given by (cf.\ \cite[Proposition~(14.3)]{Amann1990})
\begin{equation}\label{eq:particular_sol_from_wronskian}
u(x) = \sum_{s=1}^p (-1)^{p-s} \varphi^s(x) \int_{x_0}^x\psi^s(x')\,{\rm d}x'
\quad\mbox{with}\quad
\psi^s:=\frac{f}{g^p}\frac{\mathrm W(\varphi^1,\dots,\lefteqn{\varphi^s}\!\smash{\diagdown}\,,\dots,\varphi^p)}{\mathrm W(\varphi^1,\dots,\varphi^p)}.
\end{equation}
Finally, we recall the well-known fact that any $p$th order equation of the form \eqref{eq:ode_no_parameter} can be rewritten
as a linear system of first-order ordinary differential equations in the normal Cauchy form,
so that the solution theory of (scalar) equations of the form \eqref{eq:ode_no_parameter} can be reduced to that of such systems.
Concretely, setting \[v^s:=u_{s-1},\quad s=1,\dots,p,\]
%\[v^p_1 = -\frac{g^0}{g^p}v^1 -\frac{g^1}{g^p}v^2- \dots -\frac{g^{p-1}}{g^p}v^p + \frac{f}{g^p},\]
equation \eqref{eq:ode_no_parameter} is equivalent to the system
\begin{equation}\label{eq:ode_system_no_parameter}
v_1 = A(x)v + F(x)
\end{equation}
for $v=(v^1,\dots,v^p)^{\mathsf T}$, where
\begin{equation}\label{eq:matrix_system}
A = \begin{pmatrix}
0 & 1 & 0 & \dots & 0 & 0 \\
0 & 0 & 1 & \dots & 0 & 0 \\
0 & 0 & 0 & \dots & 0 & 0 \\
\vdots & \vdots & \vdots & \ddots & \vdots & \vdots \\
0 & 0 & 0 & \dots & 0 & 1 \\
-\frac{g^0}{g^p} & -\frac{g^1}{g^p} & -\frac{g^2}{g^p} & \dots & -\frac{g^{p-2}}{g^p} & -\frac{g^{p-1}}{g^p}
\end{pmatrix},
\quad
F = \begin{pmatrix}
0\\
0\\
0\\
\vdots\\
0\\
\frac{f}{g^p}
\end{pmatrix}.
\end{equation}

In this paper we address the question of how this solution theory changes
if the coefficient functions and the right hand side in \eqref{eq:ode_no_parameter}
are allowed to additionally depend on a real parameter~$t$.
Thus we shall be investigating ODEs of the form
\begin{equation}\label{eq:main_problem}
Pu \equiv \sum_{i=0}^p g^i(t,x) u_i = f(t,x),
\end{equation}
where, analogously to the above, $u_i:=\p ^i u/\p x^i$, $i=0,\dots,p$, as well as systems of the form
\begin{equation}\label{eq:main_problem_system}
v_1 = A(t,x) v + F(t,x),
\end{equation}
for $(t,x)$ varying in some open subset $\Omega$ of ${\mathbb R}^2$.
Here $g^0,\dots,g^p,f\in {\rm C}(\Omega,\mathbb R)$, $g^p(t,x)\ne0$ for all $(t,x)\in\Omega$,
$A\in {\rm C}(\Omega,\mathrm M_p(\mathbb R))$ and $F\in {\rm C}(\Omega,\mathbb R^p)$.
We are, in particular, interested in determining the influence of the topology of $\Omega$ on the structure of the solution
spaces of~\eqref{eq:main_problem} and of~\eqref{eq:main_problem_system}.
Basic examples show that solvability may completely break down already for very simple sets $\Omega$
(e.g.\ for the punctured plane, cf.\ Example~\ref{ex:1stOrderInhomLinODEWithoutSolutions} below).
Conversely, for nice enough domains, e.g.\ for rectangles, the parameter-dependent theory
is practically the same as in the single-variable case.
We want to find out which properties of the domain determine the solvability of linear parameter-dependent ODEs.
Indeed, we will completely characterize the solvability of~\eqref{eq:main_problem} and of~\eqref{eq:main_problem_system}
in terms of a topological property of $\Omega$, namely the so-called $x$-simplicity of $\Omega$,
a notion well-known from elementary integration theory (cf., e.g., \cite{Marsden&Tromba}).
In addition, we characterize the existence of fundamental sets of solutions of~\eqref{eq:main_problem}
or of fundamental matrices of~\eqref{eq:main_problem_system},
and the nonvanishing of the corresponding Wronskians,
again in terms of the $x$-simplicity of $\Omega$ (or its connected components).

As \eqref{eq:main_problem} may also be viewed as a specific kind of linear partial differential equation,
the question of existence of solutions in terms of properties of the underlying domain bears some resemblance
to notions like H\"ormander's concept of $P$-convexity \cite{Hoermander1976,Hoermander1983_2}. %Corollary 3.55 of \cite{Hoermander1976}
In particular, it is of interest to address the problems stated above also within the framework of Schwartz distributions.

In the remainder of this introduction we fix some notations and outline the content of the sections to follow.
Let us briefly comment on our choice of notation and style of presentation.
Our original motivation for studying parameter-dependent linear ODEs derives from
our desire to develop a more rigorous theory of Darboux transformations
for (1+1)-dimensional linear evolution equations
than the existing ones; cf.\ \cite{Matveev1979,Matveev&Salle1991,Popovych&Kunzinger&Ivanova2008}.
This explains why we primarily focus on scalar equations and just outline the corresponding results for systems
(contrary to the standard approach in the ODE literature).
It also justifies the notation of variables,
$x$ for the independent variable and $t$ for the parameter, as well as their order.

\begin{notation*}
Given a subset~$U$ of the $(t,x)$-plane,
the projection $U$ to the $t$-axis is denoted by~$\pr_tU$,
where $\pr_t\colon\mathbb R^2\to\mathbb R$ is the projection $(t,x)\mapsto t$,
and for each $t_0\in\pr_tU$ the set~$U_{t_0}$ is the projection of the section of~$U$ by the line $t=t_0$ to the $x$-axis,
\begin{gather*}
\pr_tU:=\{t\in\mathbb R\mid \exists\ x\in\mathbb R\colon (t,x)\in U\},\quad
\pr_tU\ni t\mapsto U_t:=\{x\in\mathbb R\mid (t,x)\in U\}.
\end{gather*}
For a function~$g\colon U\to\mathbb R$,
the expression ``$g\ne0$ on~$U$'' means that $g(z)\ne0$ for any $z\in U$.

{%\looseness=1
For an open set~$\Omega$ of the $(t,x)$-plane,
${\rm C}(\Omega)$, ${\rm C}^\infty(\Omega)$ and ${\rm C}^\omega(\Omega)$
are the spaces of continuous, smooth and real analytic functions on~$\Omega$, respectively.
${\rm C}^p_x(\Omega)$ with $p\in\mathbb N$ denotes the subspace of functions from~${\rm C}(\Omega)$
that admit derivatives with respect to~$x$ up to order~$p$, and these derivatives are continuous on~$\Omega$.
Analogously, ${\rm C}^\omega_x(\Omega)$ denotes the subspace of functions from~${\rm C}(\Omega)$
that are real analytic with respect to~$x$ and whose derivatives with respect to~$x$ are continuous on~$\Omega$.
$\DO(\Omega)$, $\DO^\infty(\Omega)$, $\DO^\omega_x(\Omega)$ and $\DO^\omega(\Omega)$ denote the sets of linear differential operators in~$x$
(hence of the form \eqref{eq:main_problem})
with coefficients from~${\rm C}(\Omega)$, ${\rm C}^\infty(\Omega)$, ${\rm C}^\omega_x(\Omega)$ and~${\rm C}^\omega(\Omega)$, respectively,
and whose leading coefficients do not vanish on the entire (open) set~$\Omega$.
$\DO_1(\Omega)$, $\DO^\infty_1(\Omega)$, $\DO^\omega_{x,1}(\Omega)$ and $\DO^\omega_1(\Omega)$
are, respectively, the subsets of operators from $\DO(\Omega)$, $\DO^\infty(\Omega)$, $\DO^\omega_x(\Omega)$ and $\DO^\omega(\Omega)$
whose leading coefficients are equal to one.
The notation $\ord P$ and $\lcoef P$ is used for the order and the leading coefficient of the operator~$P\in\DO(\Omega)$, respectively.
If $\ord P=p$, then we view the operator~$P$ as a map from ${\rm C}^p_x(\Omega)$ to~${\rm C}(\Omega)$,
and thus (classical) solutions of the equation~$\mathcal P$: $Pu=0$
belong to~${\rm C}^p_x(\Omega)$.

}

Expressions of the form $\zeta\psi$ for some functions $\zeta\in {\rm C}(\pr_t\Omega)$ and $\psi\in {\rm C}(\Omega)$
should always be interpreted as the product of~$\psi$ by the pullback of $\zeta$ to~$\Omega$
with respect to the map~$\pr_t\big|_\Omega$\,.

As already noted above, for a function~$u$ of~$(t,x)$ we set $u_i:=\p^i u/\p x^i$, $i\in\mathbb N$, $u_0:=u$,
%(i.e., the function~$u$ is assumed to be its zeroth-order derivative),
and $\p_x=\p/\p x$.
We also employ, depending upon convenience or necessity, the notation $u_x=u_1$.
By $\mathrm W(\varphi^1,\dots,\varphi^p)$ we denote
the Wronskian of functions~$\varphi^1,\dots,\varphi^p\in {\rm C}^p_x(\Omega)$ in the variable~$x$, i.e.,
$\mathrm W(\varphi^1,\dots,\varphi^p)=\det(\p_x^{s'-1}\varphi^s)_{s,s'=1,\dots,p}$.

The indices $s$, $s'$ and $s''$ run from 1 to~$p$, and summation with respect to repeated indices is always understood.
\end{notation*}

The plan of the paper is as follows: In Section \ref{sec:xSimpleSets} we introduce sets simple with respect to a variable
(in our case, $x$), and prove some basic topological properties of such sets.
In Section \ref{sec:FundamentalSolutionSetsOfLinODEsWithParameter}
we provide an appropriate notion of fundamental sets of solutions to homogeneous linear parameter-dependent ODEs
and relate it to the nonvanishing of the corresponding Wronskian.
We also characterize both concepts in terms of $x$-simplicity of (pieces of) the underlying domain~$\Omega$.
The inhomogeneous setting is studied in Section~\ref{sec:ExistenceOfSolutionsOfInhomLinODEsWithParameter},
where we derive necessary and sufficient conditions for solvability in terms of $x$-simplicity
and also quantify the `degree of non-solvability' in case some connected component of $\Omega$ fails to be $x$-simple.
In Section \ref{sec:DistrSolutionsOfLinODEsWithParameter} we then turn to the distributional setting,
singling out the relevant case of ${\rm C}^0$-semiregular distributions.
The final Section~\ref{sec:LinSystemsOfODEsWithParameter} is devoted to the study of systems of parameter-dependent linear ODEs.
In Appendix~\ref{sec:DistributionsWithVanishingPartialDerivatives}
we prove a structure theorem for distributions with vanishing partial derivatives on domains
that are simple with respect to a variable.
This is required for deriving the general form of distributional solutions
to parameter-dependent (systems of) linear ODEs on such domains
in Sections~\ref{sec:DistrSolutionsOfLinODEsWithParameter} and~\ref{sec:LinSystemsOfODEsWithParameter}.

\section{Sets simple with respect to a variable}
\label{sec:xSimpleSets}

Given a subset~$U$ of the $(t,x)$-plane, by~${\rm lb}_U$ and~${\rm ub}_U$
we denote the \emph{lower} and \emph{upper bounds of~$U$ in~$x$},
which are functions from $\pr_tU$ to $\mathbb R\cup\{-\infty,+\infty\}$
defined by
\[
{\rm lb}_U(t):=\inf U_t,\quad {\rm ub}_U(t):=\sup U_t,\quad t\in\pr_tU.
\]
In view of the inequality ${\rm lb}_U(t)\leqslant x\leqslant{\rm ub}_U(t)$ for any $t\in\pr_tU$ and any $x\in U_t$,
the functions~${\rm lb}_U$ and~${\rm ub}_U$ may attain values only from $\mathbb R\cup\{-\infty\}$ and $\mathbb R\cup\{+\infty\}$, respectively.
It is obvious that $U\subseteq\{(t,x)\in\mathbb R^2\mid t\in\pr_tU,\,{\rm lb}_U(t)\leqslant x\leqslant{\rm ub}_U(t)\}$.
See Figure~\ref{fig:ObjectRelatedToXSimplicity} that illustrates some objects related to $x$-simplicity.

\begin{lemma}\label{lem:OnOpenSetLB&UB}
If a subset~$\Omega$ of the $(t,x)$-plane is open,
then its projection $I:=\pr_t\Omega$ is an open subset of~$\mathbb R$
and the functions~$a:={\rm lb}_\Omega$ and~$b:={\rm ub}_\Omega$
are upper and lower semi-continuous, respectively.
\end{lemma}

\begin{proof}
Since for any ball contained in~$\Omega$ its projection to the $t$-axis is an open interval contained in~$I$,
it is obvious that $I$ is an open set.
Fix an arbitrary $t_0\in I$.
Then $a(t_0)<b(t_0)$ since the set~$\Omega$ is open.
If $a(t_0)\in\mathbb R$, then for any~$\varepsilon>0$ with $a(t_0)+\varepsilon<b(t_0)$,
there is a point $z_0=(t_0,x_0)\in\Omega$ with $x_0\in(a(t_0),a(t_0)+\varepsilon]$
and thus there exists a $\delta>0$ such that the ball $B_\delta(z_0)$ is contained in~$\Omega$.
Therefore, for any $t\in(t_0-\delta,t_0+\delta)$ we have $t\in I$ and $a(t)<a(t_0)+\varepsilon$.
Analogously, if $a(t_0)=-\infty$, then for an arbitrary~$N>0$ with $-N<b(t_0)$,
again there is a point $z_0=(t_0,x_0)\in\Omega$ with $x_0\in(-\infty,-N]$
and hence there exists a $\delta>0$ such that the ball $B_\delta(z_0)$ is contained in~$\Omega$.
Hence for any $t\in(t_0-\delta,t_0+\delta)$ we have $t\in I$ and $a(t)<-N$.
In total, this means that the function~$a$ is upper semi-continuous on~$I$.
The lower semi-continuity of~$b$ is proved in a similar way.
\looseness=-1
\end{proof}

In the above notation, we have
$\Omega\subseteq\{(t,x)\mid t\in I,\, a(t)<x<b(t)\}$ if the set~$\Omega$ is open.

\begin{definition}\label{def:XSimpleSet}
We call a subset~$U$ of the $(t,x)$-plane an \emph{$x$-simple set}
if the intersection of~$U$ by any line $t=t_0\in\mathbb R$ is an open interval within this line or the empty set.
\end{definition}

Equivalently, a subset~$U$ of the $(t,x)$-plane is called an $x$-simple set if
there exist a subset~$I$ of the $t$-axis and functions $a,b\colon I\to\mathbb R\cup\{-\infty,+\infty\}$
with $a(t)<b(t)$ for any $t\in I$ such that
\begin{gather}\label{eq:EquivDefOfXSimpleSet}
U=\{(t,x)\mid t\in I,\, a(t)<x<b(t)\}.
\end{gather}
Then $I=\pr_tU$, $a={\rm lb}_U$ and~$b={\rm ub}_U$.
Note that this definition of $x$-simple set is similar to
but in fact different from and, in certain sense, more general than the one used in elementary calculus
(cf., e.g., \cite[p.~341]{Marsden&Tromba}).

\begin{lemma}\label{lem:OnOpenX-SimpleRegion}
An $x$-simple subset~$\Omega$ of the $(t,x)$-plane is open if and only if
its projection $I:=\pr_t\Omega$ is an open subset of~$\mathbb R$
and its lower and upper bounds, $a:={\rm lb}_\Omega$ and~$b:={\rm ub}_\Omega$,
are upper and lower semi-continuous functions on~$I$, respectively.
Moreover, in this case there exists a smooth function~$\theta\colon I\to\mathbb R$ such that
$a(t)<\theta(t)<b(t)$ for any $t\in I$.
\end{lemma}

\begin{figure}
\centering
\includegraphics[width=.6\linewidth]{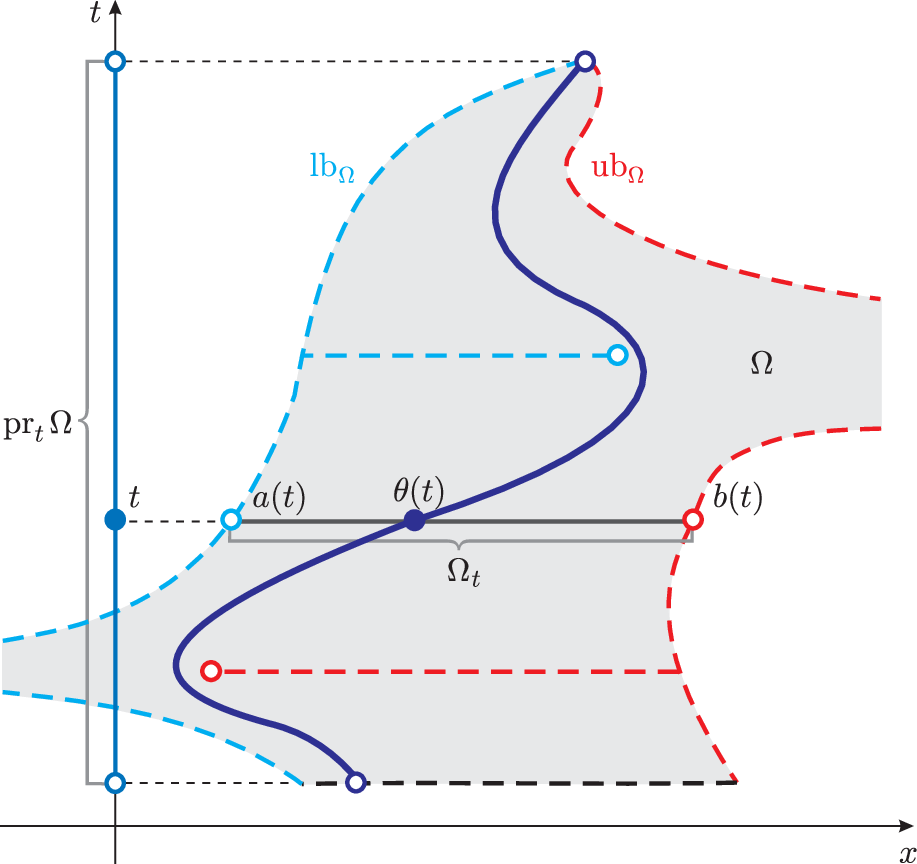}
\caption{Objects related to $x$-simplicity}\label{fig:ObjectRelatedToXSimplicity}
\end{figure}

\begin{proof}
The necessity of the first claim follows from Lemma~\ref{lem:OnOpenSetLB&UB}.
Let us prove its sufficiency.
Suppose that the set~$I$ is open and
the functions~$a$ and~$b$ are upper and lower semi-continuous, respectively.
Fix an arbitrary $(t_0,x_0)\in\Omega$.
Then $a(t_0)<x_0<b(t_0)$ and thus $\varepsilon:=\frac12\min\{x_0-a(t_0),b(t_0)-x_0,1\}>0$.
Hence there exists a neighborhood~$V$ of~$t_0$ in~$I$ such that for all $t\in V$,
$a(t)<x_0-\varepsilon$ and $b(t)>x_0+\varepsilon$,
i.e., the neighborhood~$V\times(x_0-\varepsilon,x_0+\varepsilon)$ of~$t_0$ is contained in~$\Omega$.
Therefore, the set~$\Omega$ is open.

Supposing now that $\Omega$ is $x$-simple, we can cover the open set~$I$ by a family of open subsets $(V_j)_{j\in J}$,
where $J$ is some index set, such that for any $j\in J$ there exists $c_j\in\mathbb R$
with $a(t)<c_j<b(t)$ for any $t\in V_j$.
Let $(\chi^j)_{j\in J}$ be a ${\rm C}^\infty$-partition of unity that is subordinate to $(V_j)_{j\in J}$,
i.e., $(\supp\chi^j)_{j\in J}$ is locally finite and $\supp\chi^j\subseteq V_j$ for any $j\in J$
\cite[Theorem~A.1]{MadsenTornehave1997}.
Then the function $\theta:=\sum_{j\in J}\chi^jc_j$ belongs to ${\rm C}^\infty(I)$,
and for any $t\in I$ we obtain
\[
a(t)=\sum_{j\in J}\chi^j(t)a(t)<\sum_{j\in J}\chi^j(t)c_j=\theta(t)<\sum_{j\in J}\chi^j(t)b(t)=b(t).
\]
(Here non-strict inequalities are clear for any $j\in J$,
but also for any $t\in I$ there exists $j_0\in J$ with $\chi^{j_0}(t)>0$,
and for this term we have $\chi^{j_0}(t)a(t)<\chi^{j_0}(t)c_j<\chi^{j_0}(t)b(t)$,
which implies that strict inequalities hold for the entire sums.)
\end{proof}

\begin{corollary}\label{lem:SimpleConnectnessOfX-SimpleSets}
Any $x$-simple open connected set is simply connected.
\end{corollary}

\begin{proof}
Suppose that the set~$\Omega\subseteq\mathbb R^2$ is $x$-simple, open and connected.
In view of Lemma~\ref{lem:OnOpenX-SimpleRegion}, there exists a smooth function~$\theta\in {\rm C}^\infty(\pr_t\Omega)$
whose graph is contained in~$\Omega$. (In fact, the continuity of~$\theta$ is sufficient for the further proof.)
Fix a point $t_0\in\pr_t\Omega$ and consider an arbitrary continuous path $\gamma\colon S^1\to\Omega$,
$S^1\ni\tau\mapsto(\gamma^1(\tau),\gamma^2(\tau))\in\Omega$,
where $S^1$ is the unit circle.
The path $\gamma$ can be shrunken to the point $(t_0,\theta(t_0))$ within~$\Omega$
using the map from the unit disk to~$\Omega$ that is defined~by
\[
(\rho,\tau)\mapsto
\begin{cases}
\Big(\gamma^1(\tau)+(1-2\rho)(t_0-\gamma^1(\tau)),\,\theta\big(\gamma^1(\tau)+(1-2\rho)(t_0-\gamma^1(\tau))\big)\Big),\quad \rho\in\big[0,\tfrac12\big],
\\[1.5ex]
\Big(\gamma^1(\tau),\,\gamma^2(\tau)+2(1-\rho)\big(\theta(\gamma^1(\tau))-\gamma^2(\tau)\big)\Big),\quad \rho\in\big[\tfrac12,1\big],
\end{cases}
\]
where $(\rho,\tau)$ are the `polar' coordinates on the disk, $\rho\in[0,1]$ and $\tau\in S^1$.
(Roughly speaking, we first shrink the path $\gamma$ along the $x$-direction
to the arc $\{(t,\theta(t))\mid t\in\pr_t\gamma(S^1)\}$ of the graph of the function~$\theta$
and then shrink this arc along itself to the point $(t_0,\theta(t_0))$.)
\end{proof}

Note that an open simply connected set is not in general $x$-simple.
An example of such a set is $\mathbb R^2\setminus\big([0,+\infty)\times\{0\}\big)$.

The following lemma introduces an essential technical tool for our further investigation: it identifies,
within any non-$x$-simple set, a certain configuration that will allow us to construct differential operators
on $\Omega$ with `problematic' behavior.

\begin{lemma}\label{lem:OnOpenConnectedNonX-SimpleSets}\looseness=-1
For any open connected non-$x$-simple subset~$\Omega$ of the $(t,x)$-plane,
there exist $\tilde t_0,\varepsilon,\tilde x_1,\tilde x_2\in\mathbb R$ with $\varepsilon>0$, and $\tilde x_1\leqslant\tilde x_2$
such that, up to reflections in~$t$, the set~$\Omega$
does not intersect a closed subset~$\Upsilon$ of the line segment $\{\tilde t_0\}\times[\tilde x_1,\tilde x_2]$
with $(\tilde t_0,\tilde x_1),(\tilde t_0,\tilde x_2)\in\Upsilon$,
and contains the~subset
\begin{gather}\label{eq:SpecialRectangleDomain}
[\tilde t_0-\varepsilon,\tilde t_0]\times[\tilde x_1-\varepsilon,\tilde x_2+\varepsilon]\setminus\Upsilon.
\end{gather}
\end{lemma}

\begin{proof}%\looseness=-1
Since the set~$\Omega$ is not $x$-simple, there exists $t_0\in\pr_t\Omega$ such that $\Omega_{t_0}$ is not connected,
i.e.,  for some $x_0,x_1,x_2,x_3\in\mathbb R$ with $x_0<x_1\leqslant x_2<x_3$ we have
$[x_0,x_1),(x_2,x_3]\subset\Omega_{t_0}$ and $x_1,x_2\notin\Omega_{t_0}$;
see Figure~\ref{fig:OnOpenConnectedNonX-SimpleSets}.
Since the set~$\Omega$ is connected,
there exists a (continuous) path $\gamma\colon[0,1]\to\Omega$ with $\gamma(0)=(t_0,x_0)$ and $\gamma(1)=(t_0,x_3)$.
Without loss of generality, we can assume the map~$\gamma$ injective.%
\footnote{%
Using the openness of~$\Omega$, we can additionally assume that the image of~$\gamma$ is a polygonal line,
but this is not essential for the present proof.
}
Let
\[
\tau_0=\sup\big\{\tau\in[0,1]\mid\gamma(\tau)\in\{t_0\}\times[x_0,x_1)\big\},\quad
\tau_1=\inf\big\{\tau\in[0,1]\mid\gamma(\tau)\in\{t_0\}\times(x_1,x_3]\big\}.
\]
Replacing $\gamma$ by its subpath $\gamma\big|_{[\tau_0,\tau_1]}$,
$(t_0,x_0)$ by $\gamma(\tau_0)$,
$(t_0,x_3)$ by $\gamma(\tau_1)$ and
$x_2$ by the supremum of the relative complement of $\Omega_{t_0}$ in the new interval $(x_0,x_3)$,
we can also assume that $(t_0,x_0)$ and $(t_0,x_3)$
are the only common points of $\gamma([0,1])$ with $\{t_0\}\times\mathbb R$.
We complete $\gamma([0,1])$ by $\{t_0\}\times(x_0,x_3)$ to a simple closed curve, which we denote by~$C$.
According to the Jordan curve theorem, this curve divides the $(t,x)$-plane
into the (bounded) interior~$\Gamma$ and the (unbounded) exterior~$\tilde\Gamma$.
Up to reflections in~$t$, we can assume that there exists a neighborhood~$U$ of the point $(t_0,x_1)$ such that
$U\cap\Gamma$ and $U\cap\tilde\Gamma$ contain only points with negative and positive values of $t-t_0$, respectively.
Set \[\tilde t_0=\inf\{t\in\mathbb R\mid \exists\ x\in\mathbb R\colon (t,x)\in(C\cup\Gamma)\setminus\Omega\}.\]

\begin{figure}
\centering
\includegraphics[width=1.\linewidth]{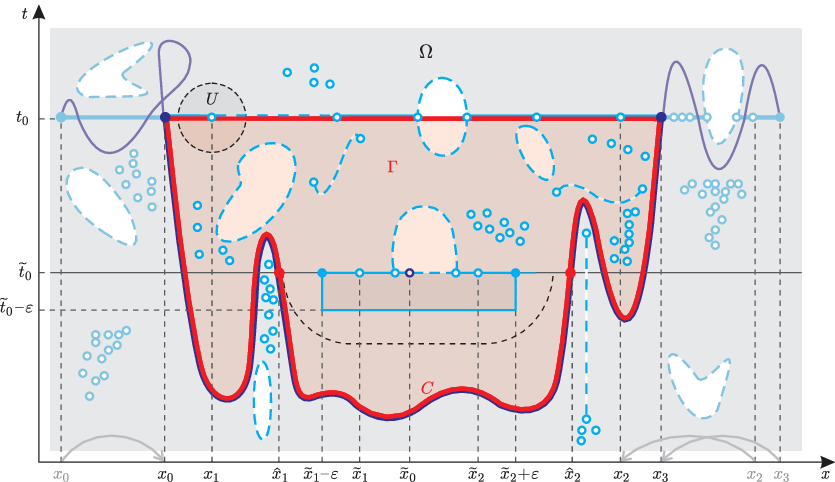}
\caption{Existence of a rectangle with a punctured boundary in an open connected non-$x$-simple set}\label{fig:OnOpenConnectedNonX-SimpleSets}
\end{figure}

If $\tilde t_0=t_0$, then we set $\tilde x_1=x_1$, $\tilde x_2=x_2$ and
$\varepsilon=\frac12\mathop{\rm dist}\big(\{\tilde t_0\}\times[\tilde x_1,\tilde x_2],\gamma([0,1])\big).$ %\frac1{\sqrt2}
(The distance is measured between disjoint compact sets and thus $\varepsilon>0$.)

\looseness=-1
Otherwise, the compactness of $(C\cup\Gamma)\setminus\Omega$ implies that
there exists $\tilde x_0\in\mathbb R$ with $(\tilde t_0,\tilde x_0)\in\Gamma\setminus\Omega$.
Since the set~$\Gamma$ is bounded, the line $t=\tilde t_0$ intersects the curve~$C$
in at least one point with $x$-coordinate less than~$\tilde x_0$ and
in at least one point with $x$-coordinate greater than~$\tilde x_0$.
Therefore the values
\begin{gather*}
\hat x_1=\sup\big\{x\in\mathbb R\mid (\tilde t_0,x)\in\gamma([0,1]),\, x<\tilde x_0\big\}, \quad
\tilde x_1=\inf\big\{x\in\mathbb R\mid (\tilde t_0,x)\notin\Omega,\, x>\hat x_1\big\},
\\
\hat x_2=\inf\big\{x\in\mathbb R\mid (\tilde t_0,x)\in\gamma([0,1]),\, x>\tilde x_0\big\}, \quad
\tilde x_2=\sup\big\{x\in\mathbb R\mid (\tilde t_0,x)\notin\Omega,\, x<\hat x_2\big\}
\end{gather*}
are well defined
as the supremum (resp.\ infimum) of  a nonempty set that is bounded from above (resp.\ below),
and $\tilde x_1\leqslant\tilde x_2$.
The line segment $\{\tilde t_0\}\times(\hat x_1,\hat x_2)$ does not intersect the curve~$C$
and contains the point $(\tilde t_0,\tilde x_0)$, which belongs to the interior~$\Gamma$.
Therefore, this segment is contained in~$\Gamma$.
The value of~$\varepsilon$ is defined as in the previous case.

The chosen values of $\tilde t_0$, $\varepsilon$, $\tilde x_1$ and $\tilde x_2$ then satisfy the claimed properties.
\end{proof}

\begin{definition}\label{def:XSimplePiece}
If there exists an open interval~$I$ of the $t$-axis such that
the intersection of a subset~$U$ of the $(t,x)$-plane by the strip $I\times\mathbb R$
has an $x$-simple connected component,
then we call this component an \emph{$x$-simple piece} of~$U$.
\end{definition}

\begin{remark}\label{rem:OnSpecificOpenSetsWithNoXSimplePieces}
Suppose that for an open set~$\Omega$ of the $(t,x)$-plane
the subset~$J$ of~$t$'s from $\pr_t\Omega$ with connected $\Omega_t$'s is dense in~$\pr_t\Omega$.
Then the set~$\Omega$ contains no $x$-simple pieces if and only if
the complement of~$J$ in $\pr_t\Omega$ is also dense in~$\pr_t\Omega$.
\end{remark}

\begin{figure}
\centering
\includegraphics[width=.5\linewidth]{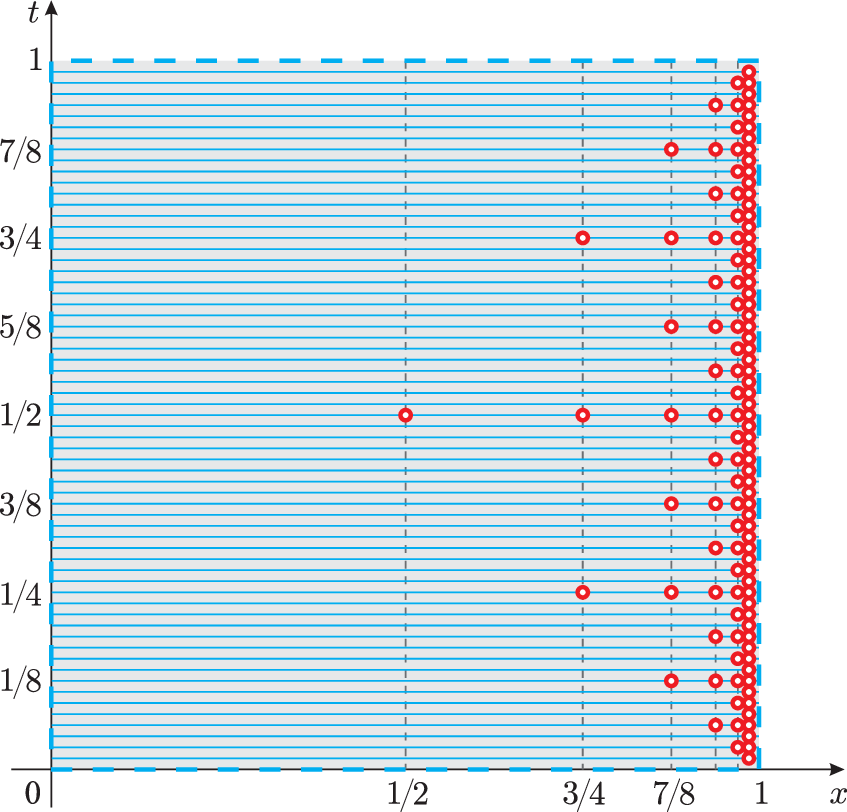}
\caption{``Infinitely punctured open square'' without $x$-simple pieces}\label{fig:InfinitelyPuncturedOpenSquare}
\end{figure}

\begin{example}\label{ex:InfinitelyPuncturedOpenSquare}
The set $\Omega:=\big((0,1)\times(0,1)\big)\setminus\big\{(2^{-k}l,1-2^{-k}),\,l=1,\dots,2^k-1,\,k\in\mathbb N\big\}$
is open, connected, and contains no $x$-simple pieces.
See Figure~\ref{fig:InfinitelyPuncturedOpenSquare}.
\end{example}

Open $x$-simple regions naturally arise in the context
of fundamental sets of solutions of linear ordinary differential equations depending on a parameter.
Moreover, several properties of such equations depend on
whether the underlying domain~$\Omega$ of the independent variable~$x$ and the parameter~$t$ is $x$-simple
and how the $x$-simplicity is combined with the connectedness,
in particular, whether all the connected components of~$\Omega$ or at least some of them are $x$-simple
or whether the domain~$\Omega$ has $x$-simple pieces.
These properties include
\begin{itemize}\itemsep=0ex
\item the existence of fundamental sets of solutions and of sets of solutions with nonvanishing Wronskians,
\item the relation between these two kinds of solution sets,
\item the existence of solutions that are not identically zero for such homogeneous equations and
\item the general existence of solutions for such inhomogeneous equations.
\end{itemize}
See Figure~\ref{fig:VariationsOnXSimplicityAndConnectedness}
for some variants of combining $x$-simplicity, connectedness and simple connectedness.

\begin{figure}
\centering
\includegraphics[width=.9\linewidth]{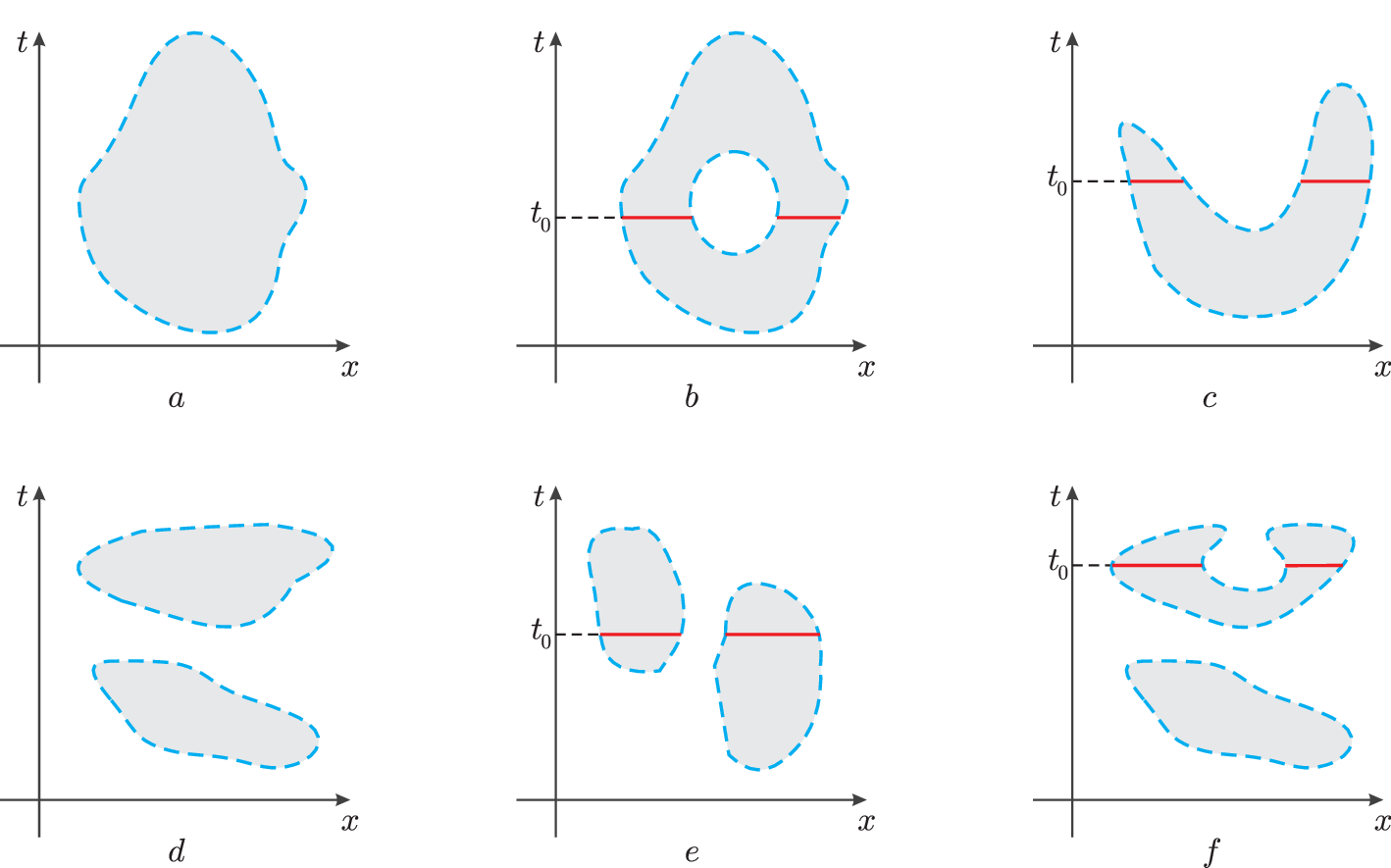}
\caption{%
Variants of combining $x$-simplicity, connectedness and simple connectedness:
({\it a\tiny\,}) connected, $x$-simple and thus simply connected set;
({\it b\tiny\,}) connected, non-$x$-simple and thus multiply connected set;
({\it c\tiny\,}) non-$x$-simple, simply connected set;
({\it d\tiny\,}) disconnected $x$-simple set (each of its connected components is necessarily $x$-simple and thus simply connected);
({\it e\tiny\,}) disconnected non-$x$-simple set whose connected components are $x$-simple and thus simply connected;
({\it f\tiny\,}) disconnected non-$x$-simple set having a non-$x$-simple connected component.
}\label{fig:VariationsOnXSimplicityAndConnectedness}
\end{figure}

\section{Fundamental sets of solutions of homogeneous linear ordinary differential equations depending on a parameter}
\label{sec:FundamentalSolutionSetsOfLinODEsWithParameter}

Given a homogeneous linear $p$th order ordinary differential equation with the independent variable~$x$ and the parameter~$t$
and with continuous coefficients defined on an open set~$\Omega\subseteq\mathbb R^2$ of~$(t,x)$,
the question is whether there exist $p$ continuous solutions%
\footnote{By this we mean solutions from $C^p_x(\Omega)$, cf.\ the notations agreed upon in Section \ref{sec:intro}.}
of this equation with nonvanishing Wronskian on~$\Omega$.
In general, the answer is negative, as is illustrated by the following example.

\begin{example}\label{ex:1stOrderHomLinODEWithVanishingWronskian}
Consider the linear homogeneous first-order ordinary differential equation
\[
\mathcal P\colon\quad u_x=\frac u{x^2+t^2}\quad\mbox{on}\quad \Omega=\mathbb R^2\setminus\{(0,0)\}.
\]
For each fixed~$t$, its general solution is
\[
u=C\exp\left(\frac1t\arctan\frac xt\right)\quad\mbox{if}\quad t\ne0, \qquad
u=C\exp\left(-\frac1x\right)\quad\mbox{if}\quad t=0,
\]
where $C$ is an arbitrary constant.
This solution is well defined on the entire $\Omega_t=\mathbb R$ if $t\ne0$,
and should be considered separately on each $x$-semiaxis, $\mathbb R_{+}$ and $\mathbb R_{-}$, if $t=0$.
The functions
\begin{gather*}
u=\zeta^+(t)\exp\left(\frac1t\arctan\frac xt\right),\quad t>0,\ x\in\mathbb R,
\\[.5ex]
u=\zeta^-(t)\exp\left(\frac1t\arctan\frac xt\right),\quad t<0,\ x\in\mathbb R,
\end{gather*}
where the parameter function~$\zeta^+$ (resp.~$\zeta^-$) runs through
${\rm C}(\mathbb R_{+})$ (resp.\ ${\rm C}(\mathbb R_{-})$\,),
represent the general solutions of the equation~$\mathcal P$
on the domains $\mathbb R_{+}\times\mathbb R$ and $\mathbb R_{-}\times\mathbb R$, respectively.
The question is whether there exists a solution of~$\mathcal P$ that is continuous and nonvanishing on the entire~$\Omega$.
Suppose that this is the case, and that $u=\varphi(t,x)$ is such a solution.
Define the function $\zeta(t):=\varphi(t,-1)$, $t\in\mathbb R$.
We have $\zeta\in {\rm C}(\mathbb R)$ and
\[
\varphi(t,x)=
\begin{cases}
\zeta(t)\exp\left(\dfrac1t\arctan\dfrac xt+\dfrac\pi{2|t|}-\dfrac{\arctan t}t\right) \quad\mbox{if}\quad t\ne0,\ x\in\mathbb R,
\\[2.5ex]
\zeta(0)\exp\left(-\dfrac1x-1\right) \quad\mbox{if}\quad t=0,\ x\in\mathbb R_-.
\end{cases}
\]
Here we use the equality
\begin{gather}\label{eq:IdentityWithArctan}
\dfrac1t\arctan\dfrac xt-\dfrac\pi{2|t|}\sgn x=-\dfrac1t\arctan\dfrac tx \quad\mbox{if}\quad t\ne0,\ x\ne0.
\end{gather}
The right hand side function is continuous on $\mathbb R^2\setminus\big(\{0\}\times[0,+\infty)\big)$
but cannot be continuously extended to~$\Omega$ if $\zeta(0)\ne0$ since for $x>0$ and $t\to0$ we obtain
\[
\zeta(t)\exp\left(\dfrac1t\arctan\dfrac xt+\dfrac\pi{2|t|}-\dfrac{\arctan t}t\right)=
\zeta(t)\exp\left(-\dfrac1t\arctan\dfrac tx+\dfrac\pi{ |t|}-\dfrac{\arctan t}t\right)\to\infty,
\]
where the sign of infinity coincides with the sign of $\zeta(0)$.
In other words, the equation~$\mathcal P$ has
\emph{no (continuous) solution that is nonzero on the entire domain}~$\Omega$.
Moreover, any solution of this equation on~$\Omega$ vanishes on the half-axis $\{0\}\times(-\infty,0)$,
and the corresponding function~$\zeta$ is $O(e^{-\pi/|t|})$ as $t\to 0$.
Consider the solution
\[
\varphi^1(t,x)=
\begin{cases}
\exp\left(\dfrac1t\arctan\dfrac xt-\dfrac\pi{2|t|}+\dfrac{\arctan t}t\right) \quad\mbox{if}\quad t\ne0,\ x\in\mathbb R,
\\[2.5ex]
\exp\left(\dfrac1x-1\right) \quad\mbox{if}\quad t=0,\ x\in\mathbb R_+,
\\[2.5ex]
0 \quad\mbox{if}\quad t=0,\ x\in\mathbb R_-.
\end{cases}
\]
Since $\varphi^1(t,1)=1$, any solution~$\varphi$ of the equation~$\mathcal P$ on~$\Omega$
can be represented as $\varphi=\tilde\zeta\varphi^1$, where $\tilde\zeta:=\varphi(t,1)\in {\rm C}(\mathbb R)$.
In this sense the function~$\varphi^1$ constitutes a fundamental set of solutions of this equation on~$\Omega$.
\end{example}

\begin{definition}\label{def:FundamentalSolutionSetOfODEWithParameter}
Given a linear ordinary differential equation~$\mathcal P$: $Pu=0$ on an open subset~$\Omega$ of the $(t,x)$-plane,
where $P\in\DO(\Omega)$ with $\ord P=p$ and $t$ plays the role of a parameter,
we say that functions $\varphi^s\in {\rm C}^p_x(\Omega)$, $s=1,\dots,p$,
satisfying this equation constitute
\begin{itemize}
\item
a \emph{fundamental set of solutions} of~$\mathcal P$ on~$\Omega$
if any solution~$u$ of~$\mathcal P$ can uniquely be represented in the form $u=\zeta^s\varphi^s$
for certain functions $\zeta^s\in {\rm C}(\pr_t\Omega)$;
\item
a \emph{locally fundamental set of solutions} of~$\mathcal P$ on~$\Omega$
if each point of~$\Omega$ has a neighbourhood $U\subseteq\Omega$
such that the restriction of any solution~$u$ of~$\mathcal P$ to~$U$, $u\big|_U$,
can uniquely be represented in the form $u\big|_U=\zeta^s\,\varphi^s\big|_U$
for certain functions $\zeta^s\in {\rm C}(\pr_tU)$.
\end{itemize}
\end{definition}

\begin{lemma}\label{lem:OnWronskianOfFundamentalSolutionSetOfODEWithParameter}
Any solutions $\varphi^s\in {\rm C}^p_x(\Omega)$, $s=1,\dots,p$, of an equation $Pu=0$ with $P\in\DO(\Omega)$ and $\ord P=p$
that satisfy the condition $\mathrm W(\varphi^1,\dots,\varphi^p)\ne0$ on~$\Omega$
constitute a locally fundamental set of solutions of this equation.
\end{lemma}

\begin{proof}
It suffices to consider a covering of~$\Omega$ by balls $U_j:=B_{\varepsilon_j}(z_j)\subseteq\Omega$, $j\in J$,
where $J$ is some index set, and $z_j=(t_j,x_j)\in\Omega$.
For an arbitrary solution~$u$ of the equation $Pu=0$ and for each of these balls,
we have the representation \smash{$u\big|_{U_j}=\zeta^{js}\varphi^s\big|_{U_j}$},
where the functions \smash{$\zeta^{js}\in C\big((t_j-\varepsilon_j,t_j+\varepsilon_j)\big)$}
are defined, for each $t\in(t_j-\varepsilon_j,t_j+\varepsilon_j)$,
as solutions of the system $\zeta^{js}(t)\varphi^s_{s'-1}(t,x_j)=u_{s'-1}(t,x_j)$, $s'=1,\dots,p$.
\end{proof}

\begin{theorem}\label{thm:OnFundSolutionSetOfLinODEsWithParameter}
Given an open subset~$\Omega$ of the $(t,x)$-plane, the following are equivalent:
\begin{itemize}\itemsep=0ex
\item[{\rm({\it  i})}] Any homogeneous linear ordinary differential equation $Pu=0$ with $P\in\DO(\Omega)$ %, where $t$ plays the role of a parameter,
admits a fundamental set of solutions on~$\Omega$ with Wronskian nonvanishing on the entire~$\Omega$.
\item[{\rm({\it ii})}] $\Omega$ is an $x$-simple region.
\end{itemize}
\end{theorem}

\begin{proof} {\rm({\it ii})}${}\Rightarrow{}${\rm({\it i})}: %Suppose that $\Omega$ is an $x$-simple region and
Consider an arbitrary $P\in\DO(\Omega)$.
In view of Lemma~\ref{lem:OnOpenX-SimpleRegion}, there exists a function $\theta\in {\rm C}^\infty(I)$ with $I=\pr_t\Omega$
such that its graph is contained in~$\Omega$.
For each $t\in I$ and $s\in\{1,\dots,p\}$, we consider the initial value problem for the equation $Pu=0$ on $\Omega_t$
with the initial conditions $u_{s'-1}=\delta_{ss'}$, $s'=1,\dots,p$, at $x=\theta(t)$ and then vary~$t$ through~$I$.
Here $\delta_{ss'}$ is the Kronecker delta.
The collection of the solutions $\varphi^s\colon\Omega\to\mathbb R$ of the above problems
then satisfies the required properties.

{\rm({\it i})}${}\Rightarrow{}${\rm({\it ii})}:
Supposing that the open set~$\Omega$ is not $x$-simple, we distinguish two cases.

First, we assume that each connected component of $\Omega$ is an $x$-simple set but the entire~$\Omega$ is not.
This means that there are connected components~$U_1$ and~$U_2$ of $\Omega$
with overlapping projections $\pr_tU_1$ and $\pr_tU_2$ to the $t$-axis.
Suppose that for some $P\in\DO(\Omega)$ of some order~$p$
the equation~$\mathcal P$: $Pu=0$ possesses a fundamental set of solutions $\varphi^1$, \dots, $\varphi^p$ on~$\Omega$.
In view of the previous part of the proof, this equation possesses sets of $p$~solutions with nonzero Wronskians
on each connected component of $\Omega$ and hence it does on the entire~$\Omega$.
Therefore the Wronskian of any fundamental set of solutions of~$\mathcal P$ does not vanish on~$\Omega$.
Using Lemma~\ref{lem:OnOpenX-SimpleRegion}, we fix a function $\theta\in {\rm C}^\infty(\pr_tU_1)$
whose graph is contained in~$U_1$.
There is a solution~$\psi$ of~$\mathcal P$ such that $\psi=1$ on this graph and $\psi=0$ on~$U_2$.
By assumption, $\psi=\zeta^s\varphi^s$ for some functions $\zeta^s\in {\rm C}(\pr_t\Omega)$.
These functions vanish on~$\pr_tU_2$ since $\psi\equiv 0$ on~$U_2$
and thus the solution~$\psi$ vanishes on~the intersection
of the strip $\{(t,x)\mid t\in\pr_tU_2,\,x\in\mathbb R^2\}$ with~$\Omega$.
But this contradicts the fact that $\psi(t,\theta(t))=1$ for $t\in\pr_tU_1\cap\pr_tU_2$.
Therefore, for any $P\in\DO(\Omega)$ with such an $\Omega$,
the equation $Pu=0$ possesses no (global) fundamental set of solutions on~$\Omega$.

Henceforth we may therefore assume that some connected component of $\Omega$ is not an $x$-simple region.
Applying Lemma~\ref{lem:OnOpenConnectedNonX-SimpleSets} to this component, we get that up to reflections in~$t$,
the set~$\Omega$ contains,
for some $\tilde t_0,\varepsilon,\tilde x_1,\tilde x_2\in\mathbb R$ with $\varepsilon>0$, $\tilde x_1\leqslant\tilde x_2$
and for some closed subset~$\Upsilon$ of the line segment $\{\tilde t_0\}\times[\tilde x_1,\tilde x_2]$
with $(\tilde t_0,\tilde x_1),(\tilde t_0,\tilde x_2)\in\Upsilon$,
the~subset $[\tilde t_0-\varepsilon,\tilde t_0]\times[\tilde x_1-\varepsilon,\tilde x_2+\varepsilon]\setminus\Upsilon$
and does not intersect~$\Upsilon$.
Consider any $P\in\DO(\Omega)$ with $G(t)\to-\infty$ as $t\to\tilde t_0{}^-$,
where%
\footnote{
If $G(t)\to+\infty$ as $t\to\tilde t_0{}^-$, it is necessary to carry out a reflection in~$x$
permuting the points $\tilde x_1-\varepsilon$ and $\tilde x_2+\varepsilon$.
}
\[
G(t):=\int_{\tilde x_1-\varepsilon}^{\tilde x_2+\varepsilon}\frac{g^{p-1}(t,x)}{g^p(t,x)}\,{\rm d}x,
\quad t\in[\tilde t_0-\varepsilon,\tilde t_0),
\]
and $g^p$ and $g^{p-1}$ denote the leading and subleading coefficients of~$P$, respectively.
An example of appropriate coefficients is given by $g^p(t,x)=1$ and~$g^{p-1}(t,x)=-c\big((x-\tilde x_1)^2+(t-\tilde t_0)^2\big)^{-1}$
with $c>0$ for $(t,x)\in\Omega$,
cf.\ Example~\ref{ex:1stOrderHomLinODEWithVanishingWronskian}.
Then $\mathrm W(\varphi^1,\dots,\varphi^p)=0$ on~$\{\tilde t_0\}\times[\tilde x_1-\varepsilon,\tilde x_1)$
for any solutions $\varphi^1$, \dots, $\varphi^p$ of the equation $Pu=0$.
Indeed, it suffices to prove this claim only for the point $(\tilde t_0,\tilde x_1-\varepsilon)$.
Supposing that it is not the case, by the Liouville--Ostrogradski formula we obtain%
$\mathrm W(\varphi^1,\dots,\varphi^p)(t,\tilde x_2+\varepsilon)
=\mathrm W(\varphi^1,\dots,\varphi^p)(t,\tilde x_1-\varepsilon)e^{-G(t)}\to\infty$, $t\to\tilde t_0{}^-$,
\noprint{
\[
\mathrm W(\varphi^1,\dots,\varphi^p)(t,\tilde x_2+\varepsilon)
=\mathrm W(\varphi^1,\dots,\varphi^p)(t,\tilde x_1-\varepsilon)e^{-G(t)}\to\infty, \quad t\to\tilde t_0{}^-,
\]
}
which contradicts the continuity of $\mathrm W(\varphi^1,\dots,\varphi^p)$ at $(\tilde t_0,\tilde x_2+\varepsilon)$.
\end{proof}

\begin{corollary}
If a connected component of an open set~$\Omega$ is not an $x$-simple region,
then for each~$p\in\mathbb N$ there exists an infinite-parameter family of equations
of the form $Pu=0$ with $P\in\DO^\omega_1(\Omega)$ of order~$p$
such that the Wronskian of any $p$ solutions of any of them vanishes
on the same line segment $\{t_0\}\times[x_1,x_2]$ contained in~$\Omega$.
\end{corollary}

\begin{proof}
We follow the proof of Theorem~\ref{thm:OnFundSolutionSetOfLinODEsWithParameter}
and consider an operator~$P\in\DO^\omega_1(\Omega)$ of the form $P=\sum_{q=0}^pg^q\p_x^q$,
where $g^p(t,x)=1$, $g^{p-1}(t,x)=-f(t,x)\big((x-\tilde x_1)^2+(t-\tilde t_0)^2\big)^{-1}$ for $(t,x)\in\Omega$,
$g^q$, $q=0,\dots,p-2$, are arbitrary elements of ${\rm C}^\omega(\Omega)$,
and $f$ is an arbitrary positive function in ${\rm C}^\omega(\Omega)$
that is separated from zero on the intersection of a neighborhood of~$(\tilde t_0,\tilde x_1)$ with~$\Omega$.
The coefficients of the Taylor expansions of~the functions~$f$ and $g^q$, $q=0,\dots,p-2$,
can serve as parameters of the family of equations $Pu=0$, which obviously has the required properties.
\end{proof}

\begin{corollary}
If each connected component of an open non-$x$-simple set~$\Omega$ is $x$-simple,
then any equation $Pu=0$ with $P\in\DO(\Omega)$ admits sets of $\ord P$ solutions with Wronskians nonvanishing on~$\Omega$
and no fundamental set of solutions on~$\Omega$.
\end{corollary}

\begin{corollary}\label{cor:WronkianOfFSS}
Given an open $x$-simple subset~$\Omega$ of the $(t,x)$-plane,
a solution set $\{\varphi^1,\dots,\varphi^p\}$
of a $p$th order linear ordinary differential equation $\mathcal P$: $Pu=0$ with $P\in\DO(\Omega)$
%, where $t$ plays the role of a parameter,
is fundamental on~$\Omega$ if and only if
the Wronskian of these solutions vanishes nowhere on~$\Omega$.
\end{corollary}

\begin{corollary}\label{cor:OnOpenSetsWithXSimplePieces}
1. If an open set~$\Omega$ has an $x$-simple piece,
then any differential equation $Pu=0$ with $P\in\DO(\Omega)$,
possesses a solution that is not identically zero on~$\Omega$.

2. If there are $x$-simple pieces of~$\Omega$ with overlapping projections to the $t$-axis,
then any equation of the above form admits no fundamental set of solutions on~$\Omega$.
\end{corollary}

\begin{proof}
1. Fix $P\in\DO(\Omega)$ with $\ord P=p$ and let $U$ be an $x$-simple piece of $\Omega$.
In view of Lemma~\ref{lem:OnOpenX-SimpleRegion}, there exists a function $\theta\in {\rm C}^\infty(\pr_tU)$
whose graph is contained in~$U$.
For each $t\in\pr_tU$, we consider the initial value problem for the equation $\mathcal P$: $Pu=0$ on~$U_t$
with the initial conditions $u_{s-1}=\chi^s(t)$, $s=1,\dots,p$, at $x=\theta(t)$ and then vary~$t$ through~$\pr_tU$.
Here $\chi^1$,~\dots, $\chi^p$ are bump functions 
(i.e., smooth functions with compact nonempty supports) on~$\pr_tU$.
The continuation of the solution of this problem by zero to~$\Omega$ gives a solution of~$\mathcal P$ on~$\Omega$ as required.

2. Suppose that there is another $x$-simple piece~$\tilde U$ of $\Omega$
such that $\pr_t U\cap\pr_t\tilde U\ne\varnothing$.
We choose a value $t_0\in\pr_t U\cap\pr_t\tilde U$ and additionally set the condition $\chi^1(t_0)\ne0$ in the above construction,
which results in a solution~$\psi\in {\rm C}^p_x(\Omega)$ of~$\mathcal P$ with $\supp\psi\subset U$ and $\psi\big(t_0,\theta(t_0)\big)\ne0$.
If the equation~$\mathcal P$ admitted a fundamental set $\{\varphi^1,\dots,\varphi^p\}$ of solutions on~$\Omega$,
where $p=\ord P$,
then the restrictions of these solutions to~$U$ and to~$\tilde U$ would form fundamental sets of solutions of~$\mathcal P$
on~$U$ and on~$\tilde U$, respectively.
Thus, Corollary~\ref{cor:WronkianOfFSS} would imply that the Wronskian of these solutions does not vanish on $U\cup\tilde U$.
Let us analyze the expansion $\psi=\zeta^s\varphi^s$, where $\zeta^s\in {\rm C}(\pr_t\Omega)$.
Since $\psi=0$ on $\tilde U$, the functions~$\zeta^s$ would vanish on~$\pr_t\tilde U$,
which contradicts the condition $\psi\big(t_0,\theta(t_0)\big)\ne0$.
\end{proof}

If an open set~$\Omega$ contains no $x$-simple pieces,
then there may exist a differential equation $Pu=0$ with $P\in\DO(\Omega)$,
possessing only the zero solution on~$\Omega$:

\begin{example}\label{ex:1stOrderODEWithNoNonzeroSolutions}
On the ``infinitely punctured open square'' $\Omega$ presented in Example~\ref{ex:InfinitelyPuncturedOpenSquare},
we consider the equation $u_x=H(t,x)u$, where
\[
H(t,x):=\sum_{k=1}^\infty\sum_{l=1}^{2^k-1}\frac{4^{-k}c_{kl}}{(x-1+2^{-k})^2+(t-2^{-k}l)^2}, \quad (t,x)\in\Omega,
\]
with positive constants~$c_{kl}$ such that $\{c_{ki},\, k,i\in\mathbb N\}$ is bounded above by a (positive) constant~$C$.
Note that $H\in {\rm C}^\omega(\Omega)$,
being a locally uniformly convergent sum of real analytic functions on~$\Omega$.
Indeed, take an arbitrary point $z_0=(t_0,x_0)\in\Omega$ and fix $\delta>0$
such that the ball $B_{2\delta}(z_0)$ is contained in~$\Omega$.
Then the series for~$H$ is dominated on~$B_\delta(z_0)$
by the convergent series \smash{$\sum_{k=1}^\infty\sum_{l=1}^{2^k-1}\delta^{-2}C4^{-k}$}.

Let $\psi$ be a solution of this equation on~$\Omega$.
According to the proof of Theorem~\ref{thm:OnFundSolutionSetOfLinODEsWithParameter},
the func\-tion~$\psi$ vanishes on the set \smash{$\bigcup_{k=1}^\infty\bigcup_{l=1}^{2^k-1}\{2^{-k}l\}\times(0,1-2^{-k})$},
which is dense in~$\Omega$.
Hence this function vanishes on the entire~$\Omega$.

Moreover, the above consideration allows us to conclude by induction that for any $p\in\mathbb N$
the equation $(\p_x-H)^pu=0$ admits only the zero solution on~$\Omega$.
\end{example}

Example~\ref{ex:1stOrderODEWithNoNonzeroSolutions} can be generalized to the following assertion.

\begin{theorem}\label{thm:NoNonzeroSolutionsForLinODEsWithParameterOnSpecialDomain}\looseness=-1
If an open set~$\Omega$ contains no $x$-simple pieces,
and the subset~$J$ of~$t$'s from $\pr_t\Omega$ with connected $\Omega_t$'s is dense in~$\pr_t\Omega$,
then for each~$p\in\mathbb N$ there exists an infinite-parameter family of equations
of the form $Pu=0$ with $P\in\DO^\omega_1(\Omega)$ of order~$p$
that possess only the zero solution on~$\Omega$.
\end{theorem}

\begin{proof} \looseness=-1
Given a set~$\Omega$ with the prescribed properties and $I:=\pr_t\Omega$,
we can consider each connected component of~$\Omega$ separately,
and thus we can assume that $\Omega$ is connected.
We define the set
\[\Theta:=\{(t,x)\in\mathbb R^2\mid t\in I,\,
{\rm lb}_\Omega(t)<x<{\rm ub}_\Omega(t)\},\]
which is open and $x$-simple with $\pr_t\Theta=I$.
(Moreover, it is the minimal $x$-simple set that contains~$\Omega$.)
We fix a function $\theta\in {\rm C}^\infty(I)$ with graph contained in~$\Theta$,
which exists in view of Lemma~\ref{lem:OnOpenX-SimpleRegion}.

\begin{figure}
\centering
\includegraphics[width=1.\linewidth]{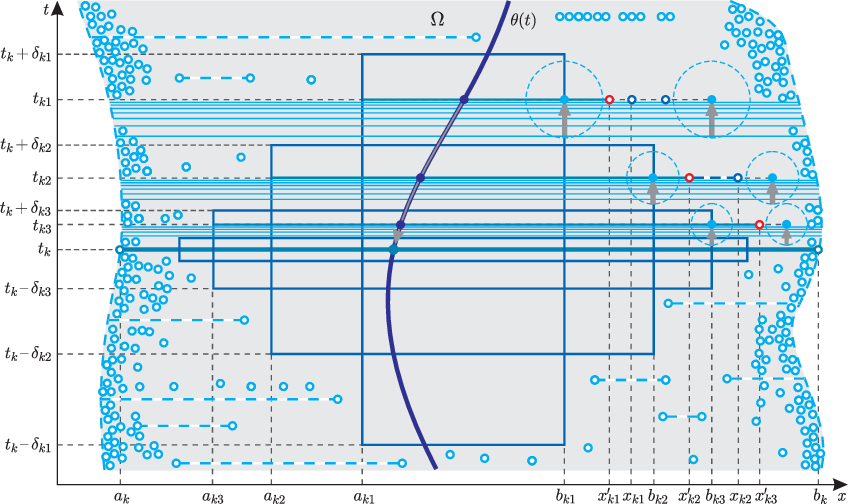}
\caption{Objects related to $t_k$ in the proof of Theorem~\ref{thm:NoNonzeroSolutionsForLinODEsWithParameterOnSpecialDomain}}
\label{fig:NoNonzeroSolutionsForLinODEsWithParameterOnSpecialDomain}
\end{figure}

We choose a countable subset $\{t_k,\,k\in\mathbb N\}$ in~$J$
that is dense in~$J$ and thus in~$I$.
For each~$k$,
we have $\Omega_{t_k}=(a_k,b_k)$,
where $a_k:={\rm lb}_\Omega(t_k)\in\mathbb R\cup\{-\infty\}$ and $b_k:={\rm ub}_\Omega(t_k)\in\mathbb R\cup\{+\infty\}$;
see Figure~\ref{fig:NoNonzeroSolutionsForLinODEsWithParameterOnSpecialDomain}.
Since the set~$\Omega$ is open, there exists a sequence of rectangles $[t_k-\delta_{ki},t_k+\delta_{ki}]\times[a_{ki},b_{ki}]$, $i\in\mathbb N$,
where $\delta_{ki}\downarrow0$, $a_{ki}\downarrow a_k$ and $b_{ki}\uparrow b_k$ strictly monotonically as $i\to\infty$,
and $a_{ki}<\theta(t)<b_{ki}$ for $t\in[t_k-\delta_{ki},t_k+\delta_{ki}]$,
and that are contained in~$\Omega$.
There exists a sequence of points $(t_{ki},x_{ki})\in\Theta\setminus\Omega$, $i\in\mathbb N$,
such that $t_{ki}\in(t_k-\delta_{ki},t_0+\delta_{ki})$ and hence either $x_{ki}<a_{ki}$ or $x_{ki}>b_{ki}$ for each $i\in\mathbb N$.
Indeed, if this was not the case for some~$i$,
then the set~$\Omega$ would possess the \mbox{$x$-simple} piece $\Omega\cap(t_k-\delta_{ki},t_0+\delta_{ki})\times\mathbb R$,
which contradicts the assumption of the lemma.
Therefore, the sequence $(x_{ki}, i\in\mathbb N)$ has limit points
that are less than~$a_k$ or greater than~$b_k$.
Define%
\footnote{
It suffices for each~$k$ to belong to a single set, either~$K_+$ or~$K_-$.
}
\[
K_+:=\bigg\{k\in\mathbb N\ \Big|\,\varlimsup_{i\to\infty}x_{ki}\geqslant b_k\bigg\},\quad
K_-:=\bigg\{k\in\mathbb N\setminus K_+\ \Big|\,      \varliminf_{i\to\infty}x_{ki}\leqslant a_k\bigg\},
\]
endowing these sets with the natural order inherited from $\mathbb N$.
Only one of them may be empty.
If $K_+\ne\varnothing$, then for each $k\in K_+$, we can assume without loss of generality
(by selecting a subsequence)
that $x_{ki}>b_{ki}$ for any $i\in\mathbb N$.
Define $x_{ki}':=\inf\{x\in(b_{ki},x_{ki}]\mid x\notin\Omega_{t_{ki}}\}$.
As a result, we construct the countable tuple $\big((t_{ki},x_{ki}'), k\in K_+,i\in\mathbb N\big)$.%
\footnote{%
In general, there may be repeated points,
but this is not essential for the further construction.
}
In a similar way,
if $K_-\ne\varnothing$, then for each $k\in K_-$, we can assume without loss of generality
that $x_{ki}<a_{ki}$ for any $i\in\mathbb N$.
Then set $x_{ki}':=\sup\{x\in[x_{ki},a_{ki})\mid x\notin\Omega_{t_{ki}}\}$.
This gives the countable tuple $\big((t_{ki},x_{ki}'),\, k\in K_-,\,i\in\mathbb N\big)$.
We define the function
\[
H(t,x):=
 \sum_{k\in K_+}\sum_{i=1}^\infty\frac{2^{-k-i}c_{ki}}{(x-x_{ki}')^2+(t-t_{ki})^2}
-\sum_{k\in K_-}\sum_{i=1}^\infty\frac{2^{-k-i}c_{ki}}{(x-x_{ki}')^2+(t-t_{ki})^2},
%\Bigg(\sum_{k\in K_+}-\sum_{k\in K_-}\Bigg)\sum_{i=1}^\infty\frac{2^{-k-i}c_{ki}}{(x-x_{ki}')^2+(t-t_{ki})^2},
\quad (t,x)\in\Omega,
\]
where the $c_{ki}$ are positive constants%
\footnote{%
These constants can be replaced by functions from ${\rm C}^\omega(\Omega)$
each of which is positive on~$\Omega$, bounded above by the same constant~$C$ on~$\Omega$
and separated from zero on the intersection of a neighborhood of the corresponding point $(t_{ki},x_{ki}')$ with~$\Omega$.
}
such that $\{c_{ki},\, k,i\in\mathbb N\}$ is bounded above.
(These $c_{ki}$ can serve as a family of infinitely many parameters, cf.\ the formulation of the proposition.)
The function~$H$ is real analytic on~$\Omega$,
which is shown similarly to Example~\ref{ex:1stOrderODEWithNoNonzeroSolutions}.
Let us prove that the equation $u_x=H(t,x)u$ possesses only the zero solution on~$\Omega$.

Any solution $\psi\in {\rm C}^p_x(\Omega)$ of this equation vanishes on all the line segments
$\{t_{ki}\}\times(a_{ki},x_{ki}')$, $k\in K_+$, $i\in\mathbb N$, and
$\{t_{ki}\}\times(x_{ki}',b_{ki})$, $k\in K_-$, $i\in\mathbb N$.
We will show this for arbitrary fixed $k\in K_+$ and $i\in\mathbb N$.
(The proof for $k\in K_-$ is similar.)
It suffices to prove that $\psi(t_{ki},b_{ki})=0$.
There exists $b_{ki}'\in\Omega_{t_{ki}}$ that is greater than~$x_{ki}'$.
Since the set~$\Omega$ is open, there exists $\delta>0$
such that both the balls $B_{2\delta}\big((t_{ki},b_{ki})\big)$ and $B_{2\delta}\big((t_{ki},b_{ki}')\big)$
are contained in~$\Omega$.
Then also $[b_{ki},b_{ki}']\subset\Omega_t$ for any $t\in J\cap[t_{ki}-\delta,t_{ki}+\delta]$.
Similarly to the proof of Theorem~\ref{thm:OnFundSolutionSetOfLinODEsWithParameter},
the assumption $\psi(t_{ki},b_{ki})\ne0$ implies that
\[
\psi(t,b_{ki}')=\psi(t,b_{ki})\exp\left(\int_{b_{ki}}^{b_{ki}'}H(t,x)\,{\rm d}x\right)\to\infty
\mbox{ \ as \ }  t\to t_{ki}
\mbox{ \ within \ } J\cap[t_{ki}-\delta,t_{ki}+\delta],
\]
which contradicts the continuity of~$\psi$ at the point $(t_{ki},b_{ki}')$.

As a result, we have $0=\psi\big(t_{ki},\theta(t_{ki})\big)\to\psi\big(t_k,\theta(t_k)\big)$ as $i\to\infty$,
and hence $\psi\big(t_k,\theta(t_k)\big)=0$.
Therefore, $\psi=0$ on the union $\bigcup_{k=1}^\infty\{t_k\}\times\Omega_{t_k}$, which is dense in~$\Omega$.
This finally implies that $\psi=0$ on~$\Omega$.

It is easy to prove by induction using the above claim on the equation $u_x=H(t,x)u$
as both the base case and a base for proving the inductive step
that for any $p\in\mathbb N$ the equation $(\p_x-H)^pu=0$ admits only the zero solution on~$\Omega$.
\end{proof}

The following is an analogue of Theorem~\ref{thm:NoNonzeroSolutionsForLinODEsWithParameterOnSpecialDomain}
for an arbitrary open subset of the $(t,x)$-plane without $x$-simple pieces
only for equations with coefficients in ${\rm C}^\omega_x(\Omega)$.

\begin{theorem}\label{thm:NoNonzeroSolutionsForLinODEsWithParameter}
An open set~$\Omega$ contains no $x$-simple pieces if and only if
for each~$p\in\mathbb N$ there exists an infinite-parameter family of equations
of the form $Pu=0$ with $P\in\DO^\omega_{x,1}(\Omega)$ of order~$p$
that possess only the zero solution on~$\Omega$.
\end{theorem}

\begin{proof}
We prove the sufficiency of the absence of $x$-simple pieces for existence of equations with only the zero solution
since the necessity follows from point~1 of Corollary~\ref{cor:OnOpenSetsWithXSimplePieces}.
Thus, suppose that an open set~$\Omega$ contains no $x$-simple pieces.

Choose a countable dense subset $\{(t^*_k,x^*_k),\,k\in\mathbb N\}$ of $\Omega$.
We consider nested open subsets $\Omega_k$, $k\in\mathbb N$, of~$\Omega$,
$\Omega_1:=\Omega\supset\Omega_2\supset\Omega_3\supset\cdots$,
and points $(t_k,x_k)\in\Omega_k$ with $(t_1,x_1):=(t^*_1,x^*_1)$
and $(t_k-t^*_k)^2+(x_k-x^*_k)^2<k^{-2}$ for $k>1$.
The subsets~$\Omega_k$ with $k>1$ will be defined recursively later.

For each $k\in\mathbb N$, we implement the following procedure.

There exists $\delta_k>0$ such that $I_k\times\{x_k\}\subset\Omega_k$, where $I_k:=(t_k-\delta_k,t_k+\delta_k)$.
Define the functions
$a^k\colon I_k\to\mathbb R\cup\{-\infty\}$ and
$b^k\colon I_k\to\mathbb R\cup\{+\infty\}$ by
$a^k(t):=\inf\{x\in\mathbb R\mid[x,x_k]\subset\Omega_t\}$ and %=\sup{(-\infty,x_k)\setminus\Omega_t}
$b^k(t):=\sup\{x\in\mathbb R\mid[x_k,x]\subset\Omega_t\}$.    %=\inf{(x_k,+\infty)\setminus\Omega_t}
%\footnote{We assume that $\sup\varnothing=-\infty$ and $\inf\varnothing=+\infty$ for the empty subset~$\varnothing$ of~$\mathbb R$.}
These functions are upper and lower semi-continuous on~$I_k$, respectively;
cf.\ the proof of Lemma~\ref{lem:OnOpenSetLB&UB}.
Indeed, fix an arbitrary $t\in I_k$.
If $a^k(t)\in\mathbb R$, then for any~$\varepsilon>0$ with $a^k(t)+\varepsilon<x_k$,
the interval $[a^k(t)+\varepsilon,x_k]$ is contained in~$\Omega_t$
and thus there exists a $\delta>0$
such that $(t-\delta,t+\delta)\subset I_k$
and $(t-\delta,t+\delta)\times[a^k(t)+\varepsilon,x_k]\subset\Omega$.
Therefore, for any $t'\in(t-\delta,t+\delta)$ we have $a^k(t')<a^k(t)+\varepsilon$.
Analogously, if $a^k(t)=-\infty$, then for an arbitrary~$N>0$ with $-N<x_k$,
the interval $[-N,x_k]$ is contained in~$\Omega_t$
and again there exists a $\delta>0$
such that  $(t-\delta,t+\delta)\subset I_k$ and $(t-\delta,t+\delta)\times[-N,x_k]\subset\Omega$.
Hence for any $t'\in(t-\delta,+\delta)$ we have $a^k(t')<-N$.
In total, this means that the function~$a^k$ is upper semi-continuous on~$I_k$.
The lower semi-continuity of~$b^k$ is proved in a similar way.

We distinguish four possible cases.
For each of Cases 2--4, we assume that the conditions of the previous cases do not hold.

\medskip\par\noindent
1.\ There exists a sequence $(t_{k0m})_{m\in\mathbb N}$ contained in~$I_k$
and strictly monotonically converging to~$t_k$ such that
\smash{$\displaystyle\beta_k:=\limsup_{m\to\infty}b^k(t_{k0m})>b^k(t_k)$} and
$(b^k(t_k),\beta_k)\cap\Omega_{t_k}\ne\varnothing$.
Set $\Lambda_k:=\{0\}$ and $t_{k0}:=t_k$.

\medskip\par\noindent
2.\ Else there exists a sequence $(t_{k0m})_{m\in\mathbb N}$ contained in~$I_k$
and strictly monotonically converging to~$t_k$ such that
\smash{$\displaystyle\alpha_k:=\liminf_{m\to\infty}a^k(t_{k0m})<a^k(t_k)$} and
$(\alpha_k,a^k(t_k))\cap\Omega_{t_k}\ne\varnothing$.
Set $\Lambda_k:=\{0\}$ and $t_{k0}:=t_k$.

\medskip\par\noindent
3.\ Else there exists a sequence $(t_{kl})_{l\in\mathbb N}$ contained in~$I_k$
and strictly monotonically converging to~$t_k$
such that for each $l\in\mathbb N$
there exists a sequence $(t_{klm})_{m\in\mathbb N}$ contained in~$I_k$
and strictly monotonically converging to $t_{kl}$ with
\smash{$\displaystyle\beta_{kl}:=\limsup_{m\to\infty}b^k(t_{klm})>b^k(t_{kl})$} and
$(b^k(t_{kl}),\beta_{kl})\cap\Omega_{t_{kl}}\ne\varnothing$.
Set $\Lambda_k:=\mathbb N$.

\medskip\par\noindent
4.\ Else there exists a sequence $(t_{kl})_{l\in\mathbb N}$ contained in~$I_k$
and strictly monotonically converging to~$t_k$
such that for each $l\in\mathbb N$
there exists a sequence $(t_{klm})_{m\in\mathbb N}$ contained in~$I_k$
and strictly monotonically converging to $t_{kl}$ with
\smash{$\displaystyle\alpha_{kl}:=\liminf_{m\to\infty}a^k(t_{klm})<a^k(t_{kl})$} and
\mbox{$(\alpha_{kl},a^k(t_{kl}))\cap\Omega_{t_{kl}}\ne\varnothing$}.
Set $\Lambda_k:=\mathbb N$.

\medskip

Let us show that one of the above cases necessarily holds.
Indeed, otherwise there exists $\delta'_k$ with $0<\delta'_k<\delta_k$
such that the restrictions of $a^k$ and $b^k$ on the interval $I'_k:=(t_k-\delta'_k,t_k+\delta'_k)$
have none of the properties associated with these cases.
Consider the intersection~$\Upsilon$ of~$\Omega$ with the strip $I'_k\times\mathbb R$
and partition it into three parts,
\begin{align*}
\Upsilon_-&:=\{(t,x)\in\Omega\mid t\in I'_k,\, x\leqslant a^k(t)\},\\
\Upsilon_0&:=\{(t,x)\in\Omega\mid t\in I'_k,\, a^k(t)<x<b^k(t)\},\\
\Upsilon_+&:=\{(t,x)\in\Omega\mid t\in I'_k,\, x\geqslant b^k(t)\};
\end{align*}
see Figure~\ref{fig:NoNonzeroSolutionsForLinODEsWithParameterOnDomainWithoutXSimplePieces}.
In fact,
$\Upsilon_0=\{(t,x)\in\mathbb R^2\mid t\in I'_k,\, a^k(t)<x<b^k(t)\}$.
From this it is obvious that $\Upsilon_0$ is a subset of~$\Omega$ that is $x$-simple and connected.
Since the lower and upper bounds of~$\Upsilon_0$ in~$x$, \smash{$a^k|_{I'_k}$} and \smash{$b^k|_{I'_k}$},
are upper and lower semi-continuous, respectively,
then Lemma~\ref{lem:OnOpenX-SimpleRegion} implies that $\Upsilon_0$ is an open set.
Hence $\Upsilon_0$ is not a connected component of~$\Upsilon$;
otherwise $\Upsilon_0$ would be an $x$-simple piece of~$\Omega$.
This implies that $\Upsilon_-\cup\Upsilon_+\ne\varnothing$ and
there exists a continuous path $\gamma=(\gamma^1,\gamma^2)\colon[0,1]\to\Upsilon$
such that $\gamma(0)\in\Upsilon_0$ and $\gamma(1)\in\Upsilon_-\cup\Upsilon_+$\,.
Define $\tau_0:=\sup\{\tau\in[0,1]\mid\gamma([0,\tau])\subset\Upsilon_0\}$.
Since the set~$\Upsilon_0$ is open, the point $\gamma(\tau_0)$ does not belong to it
and thus it belongs to $\Upsilon_-\cup\Upsilon_+$, say to $\Upsilon_+$.
It is obvious that $b^k(\gamma^1(\tau))>\gamma^2(\tau)$ for $\tau\in[0,\tau_0)$,
and $b^k(\gamma^1(\tau_0))<\gamma^2(\tau_0)$.
Therefore, for $\hat t:=\gamma^1(\tau_0)$ we have
$\hat t\in I'_k$, $\hat\beta:=\limsup_{t\to\hat t}b^k(t)>b^k(\hat t)$ and
$(b^k(\hat t),\hat\beta)\cap\Omega_{\hat t}\ne\varnothing$,
which contradicts the conditions for~$I'_k$.
\looseness=-1

\begin{figure}
\centering
\includegraphics[width=1.\linewidth]{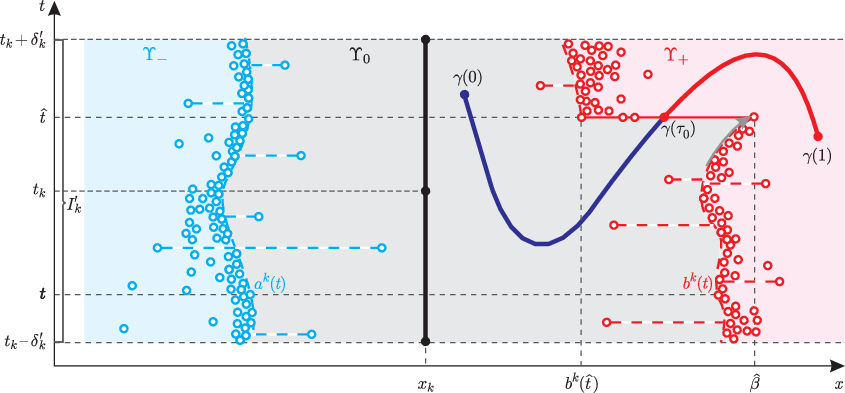}
\caption{Argumentation on the conditions for $a^k$ and~$b^k$ in the proof of Theorem~\ref{thm:NoNonzeroSolutionsForLinODEsWithParameter}}
\label{fig:NoNonzeroSolutionsForLinODEsWithParameterOnDomainWithoutXSimplePieces}
\end{figure}

Denote the set of $k$'s related to Cases~1 and~3 by~$K_+$
and the set of $k$'s related to Cases~2 and~4 by~$K_-$.
Thus, $K_+\cup K_-=\mathbb N$ and $K_+\cap K_-=\varnothing$.
We define $\Omega_1:=\Omega$ and
\[
\Omega_{k+1}=\Omega_k\setminus\big(
\{t_k,t_{kl},t_{klm},l\in\Lambda_k,m\in\mathbb N\}\times\mathbb R\big),\quad k\in\mathbb N,
\]
i.e., the set~$\Omega_{k+1}$ is obtained from~$\Omega_k$
by excluding all lines with fixed values of~$t$ that are involved in the $k$th step.
Clearly there exists a point $(t_{k+1},x_{k+1})\in\Omega_{k+1}$
with $(t_{k+1}-t^*_{k+1})^2+(x_{k+1}-x^*_{k+1})^2<(k+1)^{-2}$.
Hence the above recursion procedure is well defined.

We define the function
\[
H(t,x):=
\sum_{k\in K_+}\sum_{l\in\Lambda_k}\frac{2^{-k-l}c_{kl}\chi^k(t)}{(x-b^k(t_{kl}))^2+(t-t_{kl})^2}-
\sum_{k\in K_-}\sum_{l\in\Lambda_k}\frac{2^{-k-l}c_{kl}\chi^k(t)}{(x-a^k(t_{kl}))^2+(t-t_{kl})^2},
\]
for $(t,x)\in\Omega$,
where $c_{kl}$ are positive constants%
\footnote{%
These constants can be replaced by functions from ${\rm C}^\omega_x(\Omega)$
each of which is positive on~$\Omega$, bounded above by the same constant~$C$ on~$\Omega$
and separated from zero on the intersection of~$\Omega$ by a neighborhood of the corresponding point,
$(t_{kl},a^k(t_{kl}))$ if $k\in K_-$ or $(t_{kl},b^k(t_{kl}))$ if $k\in K_+$.
}
such that $\{c_{kl}, k\in\mathbb N, l\in\Lambda_k\}$
is bounded above by a (positive) constant~$C$,
and the functions $\chi^k\in {\rm C}^\infty(\pr_t\Omega)$ satisfy the properties
\begin{gather*}
\chi^k(t_{kl})=1,\quad \chi^k(t)\geqslant0\ \ \forall t\in\pr_t\Omega,\quad \supp \chi^k\subseteq I_k.
\end{gather*}
(Again, these $c_{ki}$ can serve as a family of infinitely many parameters, cf.\ the formulation of the theorem.)
The function~$H$ belongs to the space ${\rm C}^\omega_x(\Omega)$,%
\footnote{%
For each fixed $q\in\mathbb N$, we can obtain $q$ times continuous differentiability of~$H$ with respect to~$t$
by setting more restrictive conditions on the parameters $c_{kl}$.
More precisely, denote by ${\rm C}^{\omega,q}_{x,t}(\Omega)$ the subspace of functions in ${\rm C}^\omega_x(\Omega)$
that are continuously differentiable with respect to~$t$ $q$ times,
with each of these derivatives belonging to ${\rm C}^\omega_x(\Omega)$.
Then the above function~$H$ belongs to ${\rm C}^{\omega,q}_{x,t}(\Omega)$
if additionally $c_{kl}<C/\max\{1,|(\p^{q'}\chi^k/\p t^{q'})(t)|, t\in\pr_t\Omega, q'=1,\dots,q\}$
for all $k\in\mathbb N$ and all $l\in\Lambda_k$.
}
being a locally uniformly convergent sum of functions in ${\rm C}^\omega_x(\Omega)$.
Indeed, take an arbitrary point $z_0=(t_0,x_0)\in\Omega$ and fix $\delta>0$
such that the ball $B_{2\delta}(z_0)$ is contained in~$\Omega$.
Then the series for~$H$ is dominated on~$B_\delta(z_0)$
by the convergent series \smash{$\sum_{k=1}^\infty\sum_{l\in\Lambda_k}C\delta^{-2}2^{-k-l}$}.

Let us prove that the equation $u_x=H(t,x)u$ possesses only the zero solution on~$\Omega$.
Any solution $\psi\in {\rm C}^1_x(\Omega)$ of this equation vanishes on all the line segments
$\{t_{kl}\}\times(a^k(t_{kl}),b^k(t_{kl}))$, $k\in \mathbb N$, $l\in\Lambda_k$, and hence
on all the line segments $\{t_k\}\times(a^k(t_k),b^k(t_k))$, $k\in \mathbb N$.
We will show this for arbitrary fixed $k\in K_+$ and $l\in\Lambda_k$.
(The proof for $k\in K_-$ is similar.)
Since the union of the line segments $\{t_k\}\times(a^k(t_k),b^k(t_k))$, $k\in \mathbb N$,
is dense in~$\Omega$, this will imply that the function~$\psi$ vanishes identically on~$\Omega$.

We fix a $y_1\in(a^k(t_{kl}),b^k(t_{kl}))$ and a $y_2\in(b^k(t_{kl}),\beta_{kl})$.
Selecting a subsequence if necessary, we can assume without loss of generality
that the sequence $(b^k(t_{klm}))_{m\in\mathbb N}$ converges to~$\beta_{kl}$.
Since the function~$a^k$ is upper semi-continuous on~$I_k$ and
$b^k(t_{klm})\to\beta_{kl}>b^k(t_{kl})$ as $m\to\infty$,
there exists $N_1\in\mathbb N$ such that
$[y_1,y_2]\subset(a^k(t_{klm}),b^k(t_{klm}))\subset\Omega_{t_{klm}}$ for any $m>N_1$.
Since $\chi^k(t_{kl})=1$ and $t_{klm}\to t_{kl}$ as $m\to\infty$,
there exists $N_2\in\mathbb N$ such that $\chi^k(t_{klm})>1/2$ for any $m>N_1$.
Further we consider only values of~$m$ greater than~$N:=\max(N_1,N_2)$.
We have $\chi^{k'}(t_{klm})=0$ for any $k'>k$ and any $m\in\mathbb N$ and thus
\[
H(t_{klm},x)\geqslant\frac{2^{-k-l-1}c_{kl}}{(x-b^k(t_{kl}))^2+(t-t_{kl})^2}-C\delta_{kl}^{-2},
\]
where $\delta_{kl}:=\mathop{\rm dist}(\mathbb R\setminus I_k,\{t_{kl},t_{klm},m\in\mathbb N\})$.
Consequently, the assumption $\psi(t_{kl},y_1)\ne0$ implies that
\[
\psi(t_{klm},y_2)=\psi(t_{klm},y_1)\exp\left(\int_{y_1}^{y_2}H(t_{klm},x)\,{\rm d}x\right)\to\infty
\mbox{ \ as \ }  m\to\infty,
\]
which contradicts the continuity of~$\psi$ at the point $(t_{ki},b_{ki}')$.
Therefore, the function~$\psi$ vanishes at $(t_{kl},y_1)$
and thus it vanishes on the entire line segment $\{t_{kl}\}\times(a^k(t_{kl}),b^k(t_{kl}))$.

Similarly to the proof of Theorem~\ref{thm:NoNonzeroSolutionsForLinODEsWithParameterOnSpecialDomain},
we use the above claim on the equation $u_x=H(t,x)u$
as both the base case and a base for proving the inductive step
and derive
that for any $p\in\mathbb N$ the equation $(\p_x-H)^pu=0$ admits only the zero solution on~$\Omega$.
\end{proof}

\section{Existence of solutions of inhomogeneous linear\\ ordinary differential equations with parameter}
\label{sec:ExistenceOfSolutionsOfInhomLinODEsWithParameter}

As illustrated by the following example,
an inhomogeneous linear $p$th order ordinary differential equation with independent variable~$x$ and parameter~$t$
and with real analytic coefficients and right hand side defined on an open set~$\Omega\subseteq\mathbb R^2$ of~$(t,x)$
may possess no continuous solutions on~$\Omega$ at all.

\begin{example}\label{ex:1stOrderInhomLinODEWithoutSolutions}
Similarly to Example~\ref{ex:1stOrderHomLinODEWithVanishingWronskian},
consider the elementary linear inhomogeneous first-order ordinary differential equation
\[
\mathcal P\colon\quad u_x=\frac 1{x^2+t^2}\quad\mbox{on}\quad \Omega=\mathbb R^2\setminus\{(0,0)\},
\]
which corresponds to the operator $P:=\p_x\in\DO^\omega(\Omega)$.
For each fixed~$t$, its general solution is
\[
u=\frac1t\arctan\frac xt+C\quad\mbox{if}\quad t\ne0, \qquad
u=-\frac1x+C\quad\mbox{if}\quad t=0,
\]
where $C$ is an arbitrary constant.
This solution is well defined on the entire $\Omega_t=\mathbb R$ if $t\ne0$,
and should be separately considered on each $x$-semiaxis, $\mathbb R_{+}$ and $\mathbb R_{-}$, if $t=0$.
The functions
\begin{gather*}
u=\frac1t\arctan\frac xt+\zeta^+(t),\quad t>0,\ x\in\mathbb R,
\\[.5ex]
u=\frac1t\arctan\frac xt+\zeta^-(t),\quad t<0,\ x\in\mathbb R,
\end{gather*}
where the parameter function~$\zeta^+$ (resp.~$\zeta^-$) runs through
${\rm C}(\mathbb R_{+})$ (resp.\ ${\rm C}(\mathbb R_{-})$\,),
represent the general solutions of the equation~$\mathcal P$
on the domains $\mathbb R_{+}\times\mathbb R$ and $\mathbb R_{-}\times\mathbb R$, respectively.
The question is whether there exists a solution of~$\mathcal P$ that is continuous on the entire~$\Omega$.
Suppose that this is the case, and that $u=\varphi(t,x)$ is such a solution.
Define the function $\zeta(t):=\varphi(t,-1)$, $t\in\mathbb R$.
We have $\zeta\in {\rm C}(\mathbb R)$ and
\[
\varphi(t,x)=
\begin{cases}
\dfrac1t\arctan\dfrac xt+\dfrac\pi{2|t|}-\dfrac{\arctan t}t+\zeta(t)\quad\mbox{if}\quad t\ne0,\ x\in\mathbb R,
\\[2.5ex]
-\dfrac1x-1+\zeta(0) \quad\mbox{if}\quad t=0,\ x\in\mathbb R_-,
\end{cases}
\]
where we use the equality~\eqref{eq:IdentityWithArctan}.
Here the right hand side is continuous on $\mathbb R^2\setminus\big(\{0\}\times[0,+\infty)\big)$
but cannot be continuously extended to~$\Omega$ since for $x>0$ and $t\to0$ we obtain
\[
\dfrac1t\arctan\dfrac xt+\dfrac\pi{2|t|}-\dfrac{\arctan t}t+\zeta(t)=
-\dfrac1t\arctan\dfrac tx+\dfrac\pi{ |t|}-\dfrac{\arctan t}t+\zeta(t)\to+\infty.
\]
In other words, the equation~$\mathcal P$ has
\emph{no (continuous) solution on the entire domain}~$\Omega$.
\end{example}

\begin{example}\label{ex:InfiniteSetFamilyOf1stOrderInhomLinODEWithoutSolutions}
Generalizing Example~\ref{ex:1stOrderInhomLinODEWithoutSolutions},
consider the family of elementary linear inhomogeneous first-order ordinary differential equations
\[
\mathcal P_f\colon\quad (x^2+t^2)u_x=f(t,x)\quad\mbox{on}\quad \Omega=\mathbb R^2\setminus\{(0,0)\},
\]
with the operator $P:=(x^2+t^2)\p_x\in\DO^\omega(\Omega)$,
where the parameter function~$f$ runs through the subset~$\mathcal F$ of functions from ${\rm C}(\Omega)$
whose values at certain upper or lower half-neighborhoods of $(0,0)$ are separated from zero,
i.e., for each element~$f$ of~$\mathcal F$ there exist $\delta,\varepsilon>0$ such that,
up to reflections in~$t$ and function values,
$f(t,x)\geqslant\delta$ for $(t,x)\in(0,\varepsilon]\times[-\varepsilon,\varepsilon]$.
Supposing that the equation~$\mathcal P_f$ admits a solution~$\varphi\in {\rm C}^1_x(\Omega)$,
we obtain
\[
\begin{split}
\varphi(t,\varepsilon)&=\varphi(t,-\varepsilon)+\int_{-\varepsilon}^\varepsilon\frac{f(t,x)}{x^2+t^2}\,{\rm d}x
\geqslant\varphi(t,-\varepsilon)+\int_{-\varepsilon}^\varepsilon\frac{\delta\,{\rm d}x}{x^2+t^2}\\[.5ex]
&=\frac{\pi\delta}t-\frac{2\delta}t\arctan\dfrac t\varepsilon+\varphi(t,-\varepsilon)\to +\infty
\quad\mbox{as}\quad t\to0 \quad\mbox{within}\quad (0,\varepsilon],
\end{split}
\]
which contradicts the continuity of~$\varphi$ at $(0,\varepsilon)$.
In other words, for any~$f\in\mathcal F$ the equation~$\mathcal P_f$ has
no (continuous) solution on the entire domain~$\Omega$.
Since the set~$\mathcal F$ clearly contains infinitely many linearly independent functions, 
\emph{the quotient space ${\rm C}(\Omega)/\mathop{\rm im}P$ is infinite dimensional}.
Additionally assuming $f\in{\rm C}^\infty(\Omega)$ or $f\in{\rm C}^\omega(\Omega)$,
we also conclude that the quotient spaces
${\rm C}^\infty(\Omega)/P({\rm C}^\infty(\Omega))$ and
${\rm C}^\omega(\Omega)/P({\rm C}^\omega(\Omega))$ are infinite dimensional.
\end{example}

\begin{theorem}\label{thm:OnExistenceOfSolutionsOfInhomLinODEsWithParameter}
Given an open subset~$\Omega$ of the $(t,x)$-plane,
every inhomogeneous linear ordinary differential equation $Pu=f$ with $P\in\DO(\Omega)$ and $f\in {\rm C}(\Omega)$
%, where $t$ plays the role of a parameter,
admits solutions on the entire~$\Omega$
if and only if each connected component of~$\Omega$ is an $x$-simple set.
\end{theorem}

\begin{proof}
Without loss of generality, we may assume that the set~$\Omega$ itself is connected.

Suppose that the set~$\Omega$ is $x$-simple.
In view of Lemma~\ref{lem:OnOpenX-SimpleRegion}, there exists a function $\theta\in {\rm C}^\infty(I)$ with $I=\pr_t\Omega$
such that its graph is contained in~$\Omega$.
Consider an arbitrary $P\in\DO(\Omega)$ with $p=\ord P$.
Theorem~\ref{thm:OnFundSolutionSetOfLinODEsWithParameter} implies
that the equation $Pu=0$ admits a fundamental set of solutions on~$\Omega$
with Wronskian nonvanishing on the entire~$\Omega$, $\{\varphi^s, s=1,\dots,p\}$.
Using the Lagrange method of variation of constants, for any $f\in {\rm C}(\Omega)$
the general solution of the equation $Pu=f$ can be represented in the form $u=\psi+\sum_{s=1}^p\zeta^s\varphi^s$.
Here the tuple $(\zeta^1,\dots,\zeta^p)$ runs through ${\rm C}(I,\mathbb R^p)$
and $\psi\in {\rm C}^p_x(\Omega)$ is a particular solution of this equation that is defined~by (cf.\ \eqref{eq:particular_sol_from_wronskian})
\begin{gather}\label{eq:ParticularSolutionOfInhomLinODEsWithParameter}
\psi(t,x)=\sum_{s=1}^p\varphi^s(t,x)\int_{\theta(t)}^x\psi^s(t,x')\,{\rm d}x', \quad (t,x)\in\Omega
\end{gather}
with
\[
\psi^s:=(-1)^{p-s}\frac f{\lcoef P}\frac{ W(\varphi^1,\dots,\lefteqn{\varphi^s}\!\smash{\diagdown}\,,\dots,\varphi^p)}{\mathrm W(\varphi^1,\dots,\varphi^p)}, \quad
s=1,\dots,p.
\]
For proving $\psi\in {\rm C}^p_x(\Omega)$ it suffices to switch from the equation~\eqref{eq:main_problem}
to the equivalent linear system of first-order ordinary differential equations
in the normal form~\eqref{eq:main_problem_system} with $A$ and~$F$ defined by~\eqref{eq:matrix_system}.
%\footnote{\label{fnt:TrnsitioFromEqToSystem}%
%The standard and optimal way for proving $\psi\in {\rm C}^p_x(\Omega)$ is to switch from the equation $Pu=f$
%to the equivalent linear system of first-order ordinary differential equations,
%where derivatives of~$u$ are assume dependent variables, $v^s=u_{s-1}$, $s=1,\dots,p$.
%The system takes the form $v^s_x=v^{s+1}$, $s=1,\dots,p-1$, $v^p_x=-\sum_{s=1}^{p-1}\tilde g^{s-1}v^s+f/g^p$,
%where $\tilde g^{s-1}=g^{s-1}/g^p$, and $g^k=g^k(t,x)$ is the coefficient of $\p_x^{\,k}$ in~$P$, $k=0,\dots,p$.
%}
In other words, the equation $Pu=f$ possesses a family of solutions
that are continuous on the entire~$\Omega$ and parameterized by $p$ arbitrary continuous functions of~$t$.

Conversely, let $\Omega$ be an open set that is not $x$-simple.
In view of Lemma~\ref{lem:OnOpenConnectedNonX-SimpleSets},
there exist $\tilde t_0,\varepsilon,\tilde x_1,\tilde x_2\in\mathbb R$ with $\varepsilon>0$ and $\tilde x_1\leqslant\tilde x_2$
such that up to reflections in~$t$,
the set~$\Omega$ contains
a subset of the form $[\tilde t_0-\varepsilon,\tilde t_0]\times[\tilde x_1-\varepsilon,\tilde x_2+\varepsilon]\setminus\Upsilon$,
where $\Upsilon$ is a closed subset of $\{\tilde t_0\}\times[\tilde x_1,\tilde x_2]$
that is disjoint from~$\Omega$ and contains the points $(\tilde t_0,\tilde x_1)$ and $(\tilde t_0,\tilde x_2)$.
% with $\Upsilon\cap\Omega=\varnothing$ and $\Upsilon\ni(\tilde t_0,\tilde x_1),(\tilde t_0,\tilde x_2)$.
The equation \smash{$u_x=\big((x-\tilde x_1)^2+(t-\tilde t_0)^2\big)^{-1}$}
has no (continuous) solution on the entire domain $\Omega$;
cf.\ Example~\ref{ex:InfiniteSetFamilyOf1stOrderInhomLinODEWithoutSolutions}.
\end{proof}

If a connected component of an open set $\Omega$ is not $x$-simple,
then in fact we can show much more than just
the existence of an inhomogeneous linear ordinary differential equation $Pu=f$ with $P\in\DO(\Omega)$ and $f\in {\rm C}(\Omega)$
that possesses no continuous solutions on the entire~$\Omega$.

\begin{theorem}\label{thm:QuotientSpacesForLinODEOpsWithParameter}
If a connected component of an open set of~$\Omega$ is not $x$-simple,
then for each~$P\in\DO(\Omega)$
the quotient space ${\rm C}(\Omega)/\mathop{\rm im}P$ is infinite dimensional.
\end{theorem}

\begin{proof}
We again apply Lemma~\ref{lem:OnOpenConnectedNonX-SimpleSets} to obtain the existence, up to reflections in~$t$,
of a subset of the ``rectangular'' shape in~$\Omega$.
Below we continue to use the notation of this lemma.
There exist $\delta_k\in\mathbb R_+$, $k\in\mathbb N$, with $\delta_k\geqslant\delta_{k+1}$ for any $k\in\mathbb N$
such that
\begin{gather*}
\begin{split}
\Xi:=&(\tilde t_0-\varepsilon-\delta_1,\tilde t_0)\times(\tilde x_1-\varepsilon-\delta_1,\tilde x_2+\varepsilon+\delta_1)\\
&\cup[\tilde t_0,\tilde t_0+\delta_1)\times(\tilde x_1-\varepsilon-\delta_1,\tilde x_1-\varepsilon/2)
\cup\mathop{\cup}\limits_{k=2}^\infty[\tilde t_0,\tilde t_0+\delta_k)\times[\tilde x_1-\varepsilon/k,\tilde x_1-\varepsilon/(k+1))
\end{split}
\end{gather*}
is a subset of~$\Omega$; see Figure~\ref{fig:QuotientSpacesForLinODEOpsWithParameter}.
By construction, the set~$\Xi$ is $x$-simple and open.
In the capacity of a smooth function~$\theta$ related to~$\Xi$ according to Lemma~\ref{lem:OnOpenX-SimpleRegion}, we can choose
the constant function $\theta(t)=\tilde x_1-\varepsilon$, $t\in(\tilde t_0-\varepsilon-\delta_1,\tilde t_0+\delta_1)$.

\begin{figure}
\centering
\includegraphics[width=1.\linewidth]{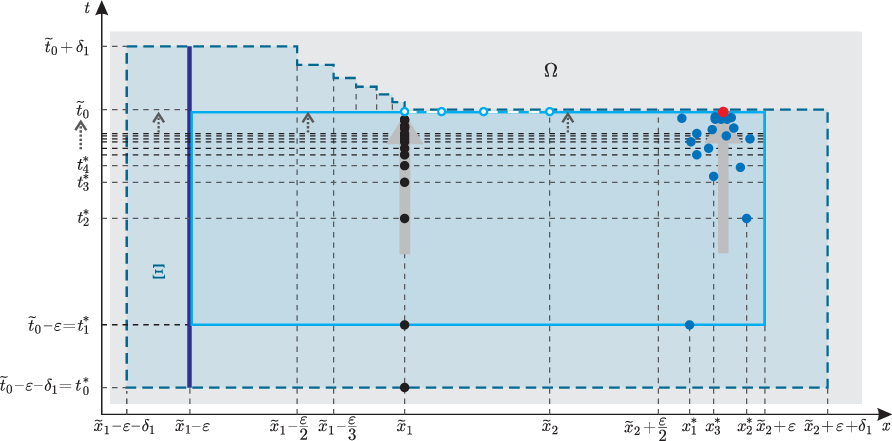}
\caption{Geometric constructions in the proof of Theorem~\ref{thm:QuotientSpacesForLinODEOpsWithParameter}}
\label{fig:QuotientSpacesForLinODEOpsWithParameter}
\end{figure}

For an arbitrary operator~$P\in\DO(\Omega)$,
we consider the corresponding inhomogeneous linear differential equations, $\mathcal P_f$: $Pu=f$ with $f\in {\rm C}(\Omega)$.
We will prove that there exist an infinite number of linearly independent continuous right hand sides~$f$ such that
for any solution~$\psi$ of~$\mathcal P_f$ on~$\Xi$
we have a sequence $\big((t^*_k,x^*_k), k\in\mathbb N\big)$ of points in~$\Xi$
with $t^*_k\to\tilde t_0$, $x^*_k\to\tilde x_3\in(\tilde x_2,\tilde x_2+\varepsilon]$ and
$\psi(t^*_k,x^*_k)\to\infty$ as $k\to\infty$.
(We will choose $t^*_k=\tilde t_0-\varepsilon/k$.)
This means that such solutions cannot be extended to (continuous) solutions of~$\mathcal P_f$ on~the entire~$\Omega$.
In other words, this implies that the equation~$\mathcal P_f$ admits no continuous solutions on~$\Omega$.

We elucidate the basic ideas of the proof by first treating the particular case of first-order differential operators.
Thus, we consider an arbitrary operator~$P$ of the form $P:=h^1(t,x)\p_x+h^0(t,x)$
with $h^0,h^1\in {\rm C}(\Omega)$ and $h^1\ne0$ on~$\Omega$.
The function $\varphi\in {\rm C}^1_x(\Xi)$ defined by
\[
\varphi(t,x):=\exp\int_{\tilde x_1-\varepsilon}^x\frac{h^0(t,x')}{h^1(t,x')}\,{\rm d}x',\quad (t,x)\in\Xi,
\]
constitutes
a fundamental set of solutions of the equation~$\mathcal P_0$ on~$\Xi$ and satisfies the initial condition
$\varphi=1$ at $x=\tilde x_1-\varepsilon$ for $t\in(\tilde t_0-\varepsilon-\delta_1,\tilde t_0+\delta_1)$.
Note that this solution is positive on~$\Xi$,
and this specific feature of first-order operators from~$\DO(\Omega)$ has no counterpart in higher orders.
We set $t^*_0:=\tilde t_0-\varepsilon-\delta_1$ and $t^*_k:=\tilde t_0-\varepsilon/k$, $k\in\mathbb N$.
For each~$k\in\mathbb N$, there exist $b_k>0$ and $\varepsilon_k$ with $0<\varepsilon_k\leqslant\varepsilon$
such that $|h^1(t^*_k,x)|\,\varphi(t^*_k,x)\leqslant b_k$
for $x\in\bar B_{\varepsilon_k}(\tilde x_1)=\{x\in\mathbb R\mid|x-\tilde x_1|\leqslant\varepsilon_k\}$.
We take a %nonnegative
continuous function $f^k$ of~$x\in\mathbb R$ and a continuous function $\chi^k$ of~$t\in\mathbb R$,
respectively, satisfying the properties
\begin{gather*}
\supp f^k\subseteq\bar B_{\varepsilon_k}(\tilde x_1), \quad f^k(x)\geqslant0\ \ \forall x\in\mathbb R, \quad
\int_{-\infty}^{+\infty}f^k(x)\,{\rm d}x\geqslant kb_k\left(1+\frac1{\varphi(t^*_k,\tilde x_2+\varepsilon)}\right),\\
\chi^k(t^*_k)=1,\quad \chi^k(t)\geqslant0\ \ \forall t\in\mathbb R,\quad \supp \chi^k\subseteq[t^*_{k-1},t^*_{k+1}]
\end{gather*}
and construct the function $f:=\left(\sum_{k=1}^\infty f^k\chi^k\right)\big|_\Omega$, which is continuous on~$\Omega$.
Any solution~$\psi$ of~$\mathcal P_f$ on~$\Xi$ can be represented in the form
\[
\psi(t,x)=\varphi(t,x)\left(\psi(t,\tilde x_1-\varepsilon)+\int_{\tilde x_1-\varepsilon}^x\frac{f(t,x')}{h^1(t,x')\varphi(t,x')}\,{\rm d}x'\right).
\]
Since $\psi\in {\rm C}^1_x(\Xi)$, it is bounded on the line segment $[\tilde t_0-\varepsilon,\tilde t_0]\times\{\tilde x_1-\varepsilon\}$,
i.e., there exists a constant $C>0$ such that $|\psi(t,\tilde x_1-\varepsilon)|\leqslant C$ for any $t\in[\tilde t_0-\varepsilon,\tilde t_0]$.
Estimating the value of~$\psi$ at $(t^*_k,\tilde x_2+\varepsilon)$ for $k>C$, we obtain
\begin{gather*}
|\psi(t^*_k,\tilde x_2+\varepsilon)|\geqslant\varphi(t^*_k,\tilde x_2+\varepsilon)
\left(\frac1{b_k}kb_k\left(1+\frac1{\varphi(t^*_k,\tilde x_2+\varepsilon)}\right)-C\right)>k,
\end{gather*}
which completes the proof for $\ord P=1$. Here $x^*_k:=\tilde x_2+\varepsilon$ for any $k\in\mathbb N$.

Now we consider the general case of $\ord P=:p$.
Following the proof of Theorem~\ref{thm:OnFundSolutionSetOfLinODEsWithParameter}
we choose the fundamental set of solutions $\{\varphi^s, s=1,\dots,p\}$ of the homogeneous equation~$\mathcal P_0$ on~$\Xi$
that satisfy the initial conditions $\varphi^s_{s'-1}=\delta_{ss'}$, $s'=1,\dots,p$,
at $x=\tilde x_1-\varepsilon$ with $t$ varying through $(\tilde t_0-\varepsilon-\delta_1,\tilde t_0+\delta_1)$.
Recall that $\delta_{ss'}$ denotes the Kronecker delta.
The Wronskian $\mathrm W:=\mathrm W(\varphi^1,\dots,\varphi^p)$ does not vanish on~$\Xi$,
and thus it does not vanish at the points $(t^*_k,\tilde x^1)$, $k\in\mathbb N$,
where again $t^*_0:=\tilde t_0-\varepsilon-\delta_1$ and $t^*_k:=\tilde t_0-\varepsilon/k$, $k\in\mathbb N$.

We fix $k\in\mathbb N$ and set
$\varphi^{k1s}:=\varphi^s(t^*_k,\cdot)\in {\rm C}^p\big((\tilde x_1-\varepsilon-\delta_1,\tilde x_2+\varepsilon+\delta_1)\big)$,
$s=1,\dots,p$.
Now choose $s_1\in\{1,\dots,p\}$ such that $|\varphi^{k1s_1}(\tilde x_1)|=\max_s|\varphi^{k1s}(\tilde x_1)|$.
This absolute value is greater than zero since $\mathrm W(t^*_k,\tilde x^1)\ne0$.
Set %$\varphi^{k2s_1}:=\varphi^{k1s_1}$, $\varphi^{k2s}:=\varphi^{k1s}-\varphi^{k1s}(\tilde x_1)\varphi^{k1s_1}/\varphi^{k1s_1}(\tilde x_1)$, $s\ne s_1$.
\[
\varphi^{k2s_1}:=\varphi^{k1s_1},\quad
\varphi^{k2s}:=\varphi^{k1s}-\frac{\varphi^{k1s}(\tilde x_1)}{\varphi^{k1s_1}(\tilde x_1)}\varphi^{k1s_1},\ \ s\ne s_1.
\]
For the transition matrix from $(\varphi^{k1s})_s$ to $(\varphi^{k2s})_s$,
its determinant equals one,
the absolute value of each of its entries is not greater than one,
and its inverse has the same properties.
Therefore, the Wronskian of~$(\varphi^{k2s})_s$ coincides with~$\mathrm W(t^*_k,\cdot)$.
Then we recursively iterate the above procedure, repeating it for ascending orders of derivatives.
More specifically, on the $s'$th step, where $s'\in\{1,\dots,p-1\}$,
we choose $s_{s'}\in N_{s'-1}:=\{1,\dots,p\}\setminus\{s_1,\dots,s_{s'-1}\}$ such that
\[|\varphi^{ks's_{s'}}_{s'-1}(\tilde x_1)|=\max\{|\varphi^{ks's}_{s'-1}(\tilde x_1)\mid s\in N_{s'-1}\}.\]
Recall that a subscript of a function denotes the corresponding number of differentiations with respect to~$x$, $f_s:=\p_x^sf$.
The above maximal absolute value is greater than zero since the Wronskian of~$(\varphi^{ks's})_s$ coincides with~$\mathrm W(t^*_k,\cdot)$
and hence it does not vanish at $\tilde x^1$.
We define
\begin{gather*}
\varphi^{k,s'+1,s_i}:=\varphi^{ks's_i},\ \ i=1,\dots,s',\\
\varphi^{k,s'+1,s}:=\varphi^{ks's}-\frac{\varphi^{ks's}_{s'-1}(\tilde x_1)}{\varphi^{ks's_{s'}}_{s'-1}(\tilde x_1)}\varphi^{ks's_{s'}},\ \ s\in N_{s'}.
\end{gather*}
For the transition matrix from $(\varphi^{ks's})_s$ to $(\varphi^{k,s'+1,s})_s$,
again its determinant equals one,
the absolute value of each of its entries is not greater than one,
and its inverse has the same properties.

The above procedure results in the functions
$\varphi^{kps}\in {\rm C}^p\big((\tilde x_1-\varepsilon-\delta_1,\tilde x_2+\varepsilon+\delta_1)\big)$, $s=1,\dots,p$,
with Wronskian coinciding with~$\mathrm W(t^*_k,\cdot)$.
Since the Wronskian~$\mathrm W$ does not vanish on~$\Xi$,
there exists $x^*_k\in[\tilde x_2+\varepsilon/2,\tilde x_2+\varepsilon]$
such that $\varphi^{kps_p}(x^*_k)\ne0$.
We also have $\varphi^{kps_i}_{j-1}(\tilde x_1)=0$, $1\leqslant j<i\leqslant p$.
Consequently, the $(p-1)$th order sub-Wronskians $\mathrm W^{ks_i}:=\mathrm W(\varphi^{kps})_{s\ne s_i}$
satisfy the conditions $\mathrm W^{ks_i}(\tilde x_1)=0$, $i=1,\dots,p-1$, and $\mathrm W^{ks_p}(\tilde x_1)\ne0$,
and hence there exist $b_k>0$ and $\varepsilon_k$ with $0<\varepsilon_k\leqslant\varepsilon$
such that
\[
\left|\frac{\mathrm W^{ks_p}(x)}{(\mathrm W \cdot\lcoef P)(t^*_k,x)}\right|\geqslant b_k,\quad
\left|\frac{\mathrm W^{ks_i}(x)}{(\mathrm W \cdot \lcoef P)(t^*_k,x)}\right|\leqslant
\frac{b_k|\varphi^{kps_p}(x^*_k)|}{4p\max_s|\varphi^{kps}(x^*_k)|},\quad
i=1,\dots,p-1,
\]
for any $x\in\bar B_{\varepsilon_k}(\tilde x_1)$.
We pick a %nonnegative
continuous function $f^k$ of~$x\in\mathbb R$ and a continuous function $\chi^k$ of~$t\in\mathbb R$,
respectively, satisfying the properties
\begin{gather*}
\supp f^k\subseteq\bar B_{\varepsilon_k}(\tilde x_1), \quad f^k(x)\geqslant0\ \ \forall x\in\mathbb R,
\\[.5ex]
\frac{2k}{b_k|\varphi^{kps_p}(x^*_k)|}\left(1+\!\sum_{s=1}^p|\varphi^{kps}(x^*_k)|\right)
\leqslant\int_{-\infty}^{+\infty}\!f^k(x)\,{\rm d}x\leqslant
\frac{4k}{b_k|\varphi^{kps_p}(x^*_k)|}\left(1+\!\sum_{s=1}^p|\varphi^{kps}(x^*_k)|\right)\!,
\\[.5ex]
\chi^k(t^*_k)=1,\quad \chi^k(t)\geqslant0\ \ \forall t\in\mathbb R,\quad \supp \chi^k\subseteq[t^*_{k-1},t^*_{k+1}]
\end{gather*}
and construct the function $f:=\left(\sum_{k=1}^\infty f^k\chi^k\right)\big|_\Omega$.
Any solution~$\psi$ of~$\mathcal P_f$ on~$\Xi$ can be represented at $t=t^*_k$ in the form
\[
\psi(t^*_k,x)=\sum_{s=1}^p\varphi^{kps}(x)\left(\,\sum_{s'=1}^pd_{kss'}\psi_{s'-1}(t^*_k,\tilde x_1-\varepsilon)
+\int_{\tilde x_1-\varepsilon}^x\frac{(-1)^{p-s}f(t,x')\mathrm W^{ks}(x')}{(\mathrm W\lcoef P)(t^*_k,x')}\,{\rm d}x'\right).
\]
Here the matrix $(d_{kss'})_{s,s'=1,\dots,p}$ is the inverse
of the Wronsky matrix $(\varphi^{kps}_{s'-1})_{s,s'=1,\dots,p}$ at $\tilde x_1-\varepsilon$,
which coincides with the transition matrix from $(\varphi^{k1s})_s$ to $(\varphi^{kps})_s$
in view of $\varphi^{k1s}_{s'-1}(\tilde x_1-\varepsilon)=\delta_{ss'}$.
Hence the set of the matrices $(d_{kss'})_{s,s'=1,\dots,p}$ with $k$ running through $\mathbb N$ is bounded.
Since $\psi\in {\rm C}^\infty(\Xi)$, the function~$\psi$ and each of its derivatives with respect to~$x$
are bounded on the line segment $[\tilde t_0-\varepsilon,\tilde t_0]\times\{\tilde x_1-\varepsilon\}$.
As a result, there exists a constant $C>0$ such that $|\sum_{s'=1}^pd_{kss'}\psi_{s'-1}(t,\tilde x_1-\varepsilon)|\leqslant C$
for any $t\in[\tilde t_0-\varepsilon,\tilde t_0]$ and for any $s\in\{1,\dots,p\}$.

Minorizing the value of~$|\psi|$ at $(t^*_k,x^*_k)$ for $k>C$,
we use the above representation for $\psi(t^*_k,x)$, compute a lower bound of the absolute value of the summand
\[
\varphi^{kps}(x^*_k)\int_{\tilde x_1-\varepsilon}^{x^*_k}\frac{(-1)^{p-s}f(t,x')\mathrm W^{ks}(x')}{(\mathrm W\lcoef P)(t^*_k,x')}\,{\rm d}x'
\]
and subtract upper bounds of the absolute values of the other summand from this lower bound.
We arrive at
\begin{gather*}
\begin{split}
|\psi(t^*_k,x^*_k)|\geqslant{}
&|\varphi^{kps_p}(x^*_k)|b_k\frac{2k}{b_k|\varphi^{kps_p}(x^*_k)|}\left(1+\sum_{s=1}^p|\varphi^{kps}(x^*_k)|\right)
\\
&-\sum_{s\ne s_p}|\varphi^{kps}(x^*_k)|\frac{b_k|\varphi^{kps_p}(x^*_k)|}{4p\max_{s'}|\varphi^{kps'}(x^*_k)|}
\frac{4k}{b_k|\varphi^{kps_p}(x^*_k)|}\left(1+\sum_{s=1}^p|\varphi^{kps}(x^*_k)|\right)
\\
&-C\sum_{s=1}^p|\varphi^{kps}(x^*_k)|
\\
{}={}&k+(k-C)\sum_{s=1}^p|\varphi^{kps}(x^*_k)|>k
\end{split}
\end{gather*}
if $k>C$.
Therefore, $|\psi(t^*_k,x^*_k)|\to+\infty$ as $k\to\infty$.
There is a convergent subsequence of the sequence $(x^*_k)_{k\in\mathbb N}^{}$,
and the limit of this subsequence belongs to the interval $[\tilde x_2+\varepsilon/2,\tilde x_2+\varepsilon]$, 
which contradicts the fact that the function~$\psi$ is continuous on~$\Omega$.
\end{proof}

An inspection of the above proof also shows the following:

\begin{corollary}
If a connected component of an open set of~$\Omega$ is not $x$-simple,
then for each~$P\in\DO^\infty(\Omega)$
the quotient space ${\rm C}^\infty(\Omega)/P({\rm C}^\infty(\Omega))$ is infinite dimensional.
\end{corollary}

\section{Distributional solutions of linear ordinary\\ differential equations with parameter}
\label{sec:DistrSolutionsOfLinODEsWithParameter}

In contrast to usual ordinary differential operators,
an operator~$P$ from $\DO^\infty(\Omega)$, where $\Omega$ is an $x$-simple open subset of~$\mathbb R^2$,
is never hypoelliptic.
At the same time, for any $f\in {\rm C}^\infty(\Omega)$
we can represent the general distributional solution of the equation $Pu=f$ on~$\Omega$
in terms of a fundamental set of smooth solutions of this equation.

\begin{proposition}\label{pro:DistrSolutionsOfLinODEsWithParameter}
Given an $x$-simple open subset~$\Omega$ of~$\mathbb R^2$,
an arbitrary $P\in\DO^\infty(\Omega)$ of order $p\in\mathbb N$ and an arbitrary $f\in {\rm C}^\infty(\Omega)$,
the general solution of the equation $Pu=f$ in $\mathcal D'(\Omega)$ can be represented in the~form
$u=T_\psi+\varphi^s\cdot(\zeta^s\otimes T_{\mathbf 1_{\mathbb R}})\big|_\Omega$,
where $T_\psi$ and $T_{\mathbf 1_{\mathbb R}}$ are the regular distributions
associated with a particular solution $\psi\in {\rm C}^\infty(\Omega)$ of this equation
and with the indicator function~$\mathbf 1_{\mathbb R}$ of~$\mathbb R$, respectively,
$\{\varphi^s, s=1,\dots,p\}$ is a fundamental set of smooth solutions of this equation on~$\Omega$
and each $\zeta^s$ runs through $\mathcal D'(\pr_t\Omega)$.
\end{proposition}

The proof of this proposition follows from
Proposition~\ref{pro:DistrSolutionsOfLinSystemsOfODEsWithParameter} below, %c.f.\ footnote~\ref{fnt:TrnsitioFromEqToSystem}.
using the equivalence of~\eqref{eq:main_problem} and~\eqref{eq:main_problem_system} via~\eqref{eq:matrix_system}.
Proposition~\ref{pro:DistrSolutionsOfLinODEsWithParameter} can be generalized to right hand sides of lower regularity.
Thus, for $f\in {\rm C}(\Omega)$ or $f\in\mathcal D'(\Omega)$
it suffices to replace the condition $\psi\in {\rm C}^\infty(\Omega)$ by the condition $\psi\in {\rm C}^p_x(\Omega)$
or $T_\psi$ by $\psi\in\mathcal D'(\Omega)$,
respectively.
For $f\in{\rm C}^0_x(\mathcal D'_t)(\Omega)$ we should substitute the condition $\psi\in{\rm C}^p_x(\mathcal D'_t)(\Omega)$,
which leads to a result in the spirit of~\cite[Theorem~4.4.8]{Hoermander1983}.
Here ${\rm C}^0_x(\mathcal D'_t)(\Omega)$ is the space of distributions on~$\Omega$ that are ${\rm C}^0$-semiregular in~$x$.
We call an element~$u$ of $\mathcal D'(\Omega)$ \emph{${\rm C}^0$-semiregular in~$x$}
if for any open rectangle $I\times J\subset\Omega$
we have that $u$ restricted to $I\times J$ is in ${\rm C}(J,\mathcal D'(I))$, cf.~\cite{Malgrange&Garding1961,Schwartz1957}.
The space ${\rm C}^p_x(\mathcal D'_t)(\Omega)$ is defined analogously.

\begin{example}\label{ex:1stOrderODEWithNonzeroDistrSolutions}
Although the equation $u_x=H(t,x)u$ from Example~\ref{ex:1stOrderODEWithNoNonzeroSolutions} has no nonzero smooth solutions on~$\Omega$,
it admits nonzero distributional solutions on this set, for example
$u=\delta_{t_0}\otimes T_\eta$,
where $\delta_{t_0}$ is the Dirac delta function at~$t_0\in(0,1)\setminus\{2^{-k}l,\,l=1,\dots,2^k-1,\,k\in\mathbb N\}$,
and $T_\eta$ is the regular distribution associated with the smooth function $\eta\in {\rm C}^\infty\big((0,1)\big)$
defined by $\eta(x):=\exp\big(\int_{1/4}^xH(t_0,x')\,{\rm d}x'\big)$, $x\in(0,1)$.
It is obvious that the Dirac delta function can be replaced by an arbitrary linear combination of its derivatives.
The equation constructed in the proof of Theorem~\ref{thm:NoNonzeroSolutionsForLinODEsWithParameterOnSpecialDomain}
also possesses similar nonzero distributional solutions.
\end{example}

As the previous example suggests, $\mathcal D'(\Omega)$ is not necessarily the best choice for seeking
solutions of equations
of the form $Pu=f$ with $P\in\DO^\infty(\Omega)$ and $f\in {\rm C}^\infty(\Omega)$
since for this space one can in fact solve such equations on each slice $\Omega_t$, $t\in\pr_t\Omega$ separately,
and slice solutions do not affect each other.
It is more natural to look for solutions in the space ${\rm C}^0_t(\mathcal D'_x)(\Omega)$
of distributions on~$\Omega$ that are ${\rm C}^0$-semiregular in~$t$.
Modifying the definition of the semiregularity in~$x$ by permuting~$t$ and~$x$,
we call an element~$u$ of $\mathcal D'(\Omega)$ \emph{${\rm C}^0$-semiregular in~$t$}
if for any open rectangle $I\times J\subset\Omega$
we have that $u$ restricted to $I\times J$ is in ${\rm C}(I,\mathcal D'(J))$.
Analogously, we can also define distributions on~$\Omega$ that are ${\rm C}^q$-semiregular in~$t$ with $q\in\mathbb N$
or ${\rm C}^\infty$-semiregular in~$t$.
\looseness=-1

\begin{proposition}\label{pro:SemiregularDistrSolutionsOfLinODEsWithParameter}
Given an $x$-simple open subset~$\Omega$ of~$\mathbb R^2$,
an arbitrary $P\in\DO^\infty(\Omega)$ and an arbitrary $f\in {\rm C}(\Omega)$ (resp.\ $f\in {\rm C}^\infty(\Omega)$),
any solution of the equation $Pu=f$ in ${\rm C}^0_t(\mathcal D'_x)(\Omega)$ (resp.\ ${\rm C}^\infty_t(\mathcal D'_x)(\Omega)$)
is a regular distribution associated with a classical (resp.\ smooth) solution  of this equation.
\end{proposition}

Again, it is better to carry out the proof for the linear system of first-order ordinary differential equations
that is equivalent to the equation $Pu=f$ via \eqref{eq:matrix_system},
see Proposition~\ref{pro:SemiregularDistrSolutionsOfLinSystemsOfODEsWithParameter} below.

\section{Parameter-dependent linear systems\\ of ordinary differential equations}
\label{sec:LinSystemsOfODEsWithParameter}

The results of the previous sections on single parameter-dependent linear ordinary differential equations
can easily be extended to parameter-dependent linear systems of ordinary differential equations in the canonical form.
Any such system is equivalent to
a linear system of first-order ordinary differential equations in the normal Cauchy form (cf.\ Section \ref{sec:intro}).
It is then evident that the results of Sections \ref{sec:FundamentalSolutionSetsOfLinODEsWithParameter},
\ref{sec:ExistenceOfSolutionsOfInhomLinODEsWithParameter} and \ref{sec:DistrSolutionsOfLinODEsWithParameter}
have direct analogues for the system case. Except for the case of distributional solutions we shall
therefore omit the proofs and confine ourselves to stating the results below.
%We present assertions for systems of the latter kind, omitting the corresponding proofs.

Let $\mathrm M_p(\mathbb R)$ denote the set of $p\times p$ matrices with real coefficients.
Consider a system of linear ordinary differential equations~$\mathcal P$: $v_x=A(t,x)v$
on an open subset~$\Omega$ of the $(t,x)$-plane,
where $A\in {\rm C}(\Omega,\mathrm M_p(\mathbb R))$,
$v=(v^1,\dots,v^p)^{\mathsf T}$ is the unknown vector function of~$(t,x)$,
$x$ is the independent variable and  $t$ plays the role of a parameter.
This system can be interpreted as a vector differential equation.
Its matrix counterpart, $M_x=A(t,x)M$ with $M\in\mathrm M_p(\mathbb R)$,
is denoted by $\mathcal P_{\rm m}$.
We assume that (classical) solutions of the system~$\mathcal P$ and the matrix equation $\mathcal P_{\rm m}$
belong to ${\rm C}^1_x(\Omega,\mathbb R^p)$ and ${\rm C}^1_x(\Omega,\mathrm M_p(\mathbb R))$, respectively.

\begin{definition}\label{def:FundamentalMatrixOfLinSystemsOfODEsWithParameter}
We say that a matrix-valued function $\Phi\in {\rm C}^1_x(\Omega,\mathrm M_p(\mathbb R))$
satisfying the equation~$\mathcal P_{\rm m}$, $\Phi_x=A(t,x)\Phi$, is
\begin{itemize}\itemsep=0ex
\item
a \emph{fundamental matrix} of~$\mathcal P$ on~$\Omega$
if any solution~$v$ of~$\mathcal P$ can uniquely be represented in the form $v=\Phi\zeta$
for some function $\zeta\in {\rm C}(\pr_t\Omega,\mathbb R^p)$;
\item
a \emph{locally fundamental matrix} of~$\mathcal P$ on~$\Omega$
if each point of~$\Omega$ has a neighbourhood $U\subseteq\Omega$
such that the restriction of any solution~$v$ of~$\mathcal P$ to~$U$, $v\big|_U$,
can uniquely be represented in the form $v\big|_U=\Phi\big|_U\zeta$
for some function $\zeta\in {\rm C}(\pr_tU,\mathbb R^p)$.
\end{itemize}
\end{definition}

\begin{lemma}\label{lem:OnMatrixSolutionsWithNonzeroDet}
Any solution $\Phi$ of the matrix equation $P_{\rm m}$ with determinant nonvanishing on~$\Omega$
is a locally fundamental matrix of the system~$\mathcal P$.
\end{lemma}

\begin{theorem}\label{thm:OnFundamentalMatrixOfLinSystemsOfODEsWithParameter}
Given an open subset~$\Omega$ of the $(t,x)$-plane, the following are equivalent:
\begin{itemize}\itemsep=0ex
\item[{\rm({\it i})}] Any homogeneous linear system of first-order ordinary differential equations $v_x=A(t,x)v$ with $A\in {\rm C}(\Omega,\mathrm M_p(\mathbb R))$,
where $t$ plays the role of a parameter,
admits a fundamental matrix on~$\Omega$ with determinant nonvanishing on the entire~$\Omega$.
\item[{\rm({\it ii})}] $\Omega$ is an $x$-simple region.
\end{itemize}
\end{theorem}

\begin{corollary}
If a connected component of an open set of~$\Omega$ is not an $x$-simple region,
then for each~$p\in\mathbb N$ there exists an infinite-parameter family of $p\times p$ matrix equations
of the form $M_x=A(t,x)M$ with $A\in {\rm C}^\omega(\Omega,\mathrm M_p(\mathbb R))$
such that the determinant of any solution of each of them vanishes
on the same line segment $\{t_0\}\times[x_1,x_2]$ contained in~$\Omega$.
\end{corollary}

\begin{corollary}
If each connected component of an open non-$x$-simple set~$\Omega$ is $x$-simple,
then any $p$-vector equation of the form $v_x=A(t,x)v$ with $A\in {\rm C}(\Omega,\mathrm M_p(\mathbb R))$
admits no fundamental matrix on~$\Omega$
although the associated matrix equation has solutions with determinants nonvanishing on~$\Omega$.
\end{corollary}

\begin{corollary}\label{cor:DetOfFundMatrix}
Given an open $x$-simple subset~$\Omega$ of the $(t,x)$-plane,
a solution~$\Phi$ of a $p\times p$ matrix equations
of the form $M_x=A(t,x)M$ with $A\in {\rm C}(\Omega,\mathrm M_p(\mathbb R))$
is a fundamental matrix on~$\Omega$ for the associated vector equation $v_x=A(t,x)v$
if and only if the determinant of~$\Phi$ does not vanish on~$\Omega$.
\end{corollary}

\begin{corollary}\label{cor:SystemsOnOpenSetsWithXSimplePieces}
1.\ If an open set~$\Omega$ has an $x$-simple piece,
then any system of differential equations $v_x=A(t,x)v$ with $A\in {\rm C}(\Omega,\mathrm M_p(\mathbb R))$,
possesses a solution that is not identically zero on~$\Omega$.

2.\ If there are $x$-simple pieces of~$\Omega$ with overlapping projections to the $t$-axis,
then any system of the above form admits no fundamental matrix on~$\Omega$.
\end{corollary}

\begin{proposition}\label{pro:NoNonzeroSolutionsForLinSystemsOfODEsWithParameterOSpecialDomain}\looseness=-1
If an open set~$\Omega$ contains no $x$-simple pieces,
and the subset~$J$ of~$t$'s from $\pr_t\Omega$ with connected $\Omega_t$'s is dense in~$\pr_t\Omega$,
then for each~$p\in\mathbb N$ there exists an infinite-parameter family of $p$-vector equations
of the form $v_x=Av$ with $A\in {\rm C}^\omega(\Omega,\mathrm M_p(\mathbb R))$
that possess only the zero solution on~$\Omega$.
\end{proposition}

\begin{theorem}\label{thm:NoNonzeroSolutionsForLinSystemsOfODEsWithParameter}
An open set~$\Omega$ contains no $x$-simple pieces if and only if
for each~$p\in\mathbb N$ there exists an infinite-parameter family of $p$-vector equations
of the form $v_x=Av$ with $A\in {\rm C}^\omega_x(\Omega,\mathrm M_p(\mathbb R))$
that possess only the zero solution on~$\Omega$.
\end{theorem}

\begin{theorem}\label{thm:OnExistenceOfSolutionsOfInhomLinSystemsOfODEsWithParameter}
Given an open subset~$\Omega$ of the $(t,x)$-plane,
every inhomogeneous linear system of first-order ordinary differential equations $v_x=Av+F$
with $A\in {\rm C}(\Omega,\mathrm M_p(\mathbb R))$ and $F\in {\rm C}(\Omega,\mathbb R^p)$,
where $t$ plays the role of a parameter, admits continuous solutions on the entire~$\Omega$
if and only if each connected component of~$\Omega$ is an $x$-simple set.
\end{theorem}

\begin{theorem}
If a connected component of an open set of~$\Omega$ is not $x$-simple,
then for each~$A\in {\rm C}(\Omega,\mathrm M_p(\mathbb R))$
the quotient space ${\rm C}(\Omega,\mathbb R^p)/(\p_x-A)({\rm C}^1_x(\Omega,\mathbb R^p))$ is infinite dimensional.
\end{theorem}

\begin{corollary}
If a connected component of an open set of~$\Omega$ is not $x$-simple,
then for each~$A\in {\rm C}^\infty(\Omega,\mathrm M_p(\mathbb R))$
the quotient space ${\rm C}^\infty(\Omega,\mathbb R^p)/(\p_x-A)({\rm C}^\infty(\Omega,\mathbb R^p))$ is infinite dimensional.
\end{corollary}

Finally, we turn to the case of distributional solutions of parameter-dependent systems:

\begin{proposition}\label{pro:DistrSolutionsOfLinSystemsOfODEsWithParameter}
Given an $x$-simple open subset~$\Omega$ of~$\mathbb R^2$,
an arbitrary $A\in {\rm C}^\infty(\Omega,\mathrm M_p(\mathbb R))$ and an arbitrary $F\in {\rm C}^\infty(\Omega,\mathbb R^p)$,
the general solution of the system $v_x=Av+F$ in $\mathcal D'(\Omega,\mathbb R^p)$
can be represented in the~form
$v=T_\psi+\Phi\cdot(\zeta\otimes T_{\mathbf 1_{\mathbb R}})\big|_\Omega$,
where the tensor product is understood componentwise,
$T_\psi$ and $T_{\mathbf 1_{\mathbb R}}$ are the regular distributions
associated with a particular solution $\psi\in {\rm C}^\infty(\Omega,\mathbb R^p)$ of this system
and with the indicator function~$\mathbf 1_{\mathbb R}$ of~$\mathbb R$, respectively,
$\Phi$ is a smooth fundamental matrix of this system on~$\Omega$,
and $\zeta$ runs through $\mathcal D'(\pr_t\Omega,\mathbb R^p)$.
\end{proposition}

\begin{proof}
We fix a particular solution $\psi\in {\rm C}^\infty(\Omega,\mathbb R^p)$ of the system $v_x=Av+F$
and a smooth fundamental matrix $\Phi$ of this system on~$\Omega$
(cf.\ Theorem \ref{thm:OnFundamentalMatrixOfLinSystemsOfODEsWithParameter}).
If a distribution $v\in\mathcal D'(\Omega,\mathbb R^p)$ satisfies the system $v_x=Av+F$,
then the distribution $\tilde v:=\Phi^{-1}(v-T_\psi)$ satisfies the system \mbox{$\tilde v_x=0$}.
By Theorem \ref{thm:distributional_derivative_tensor_product_y_simple},
the general distributional solution of the latter system is
\mbox{$\tilde v=(\zeta\otimes T_{\mathbf 1_{\mathbb R}})\big|_\Omega$},
where $\zeta$ runs through $\mathcal D'(\pr_t\Omega,\mathbb R^p)$.
\end{proof}

Similarly to Proposition~\ref{pro:DistrSolutionsOfLinODEsWithParameter},
Proposition~\ref{pro:DistrSolutionsOfLinSystemsOfODEsWithParameter} can be extended to right hand sides of lower regularity.
Thus, for $F\in {\rm C}(\Omega,\mathbb R^p)$ or $F\in\mathcal D'(\Omega,\mathbb R^p)$
it suffices to replace the condition $\psi\in {\rm C}^\infty(\Omega,\mathbb R^p)$ by the condition $\psi\in {\rm C}^1_x(\Omega,\mathbb R^p)$
or $T_\psi$ by $\psi\in\mathcal D'(\Omega,\mathbb R^p)$,
respectively.

\begin{proposition}\label{pro:SemiregularDistrSolutionsOfLinSystemsOfODEsWithParameter}
Given an $x$-simple open subset~$\Omega$ of~$\mathbb R^2$,
an arbitrary $A\in {\rm C}^\infty(\Omega,\mathrm M_p(\mathbb R))$ and an arbitrary $F\in {\rm C}(\Omega,\mathbb R^p)$ (resp.\ $F\in {\rm C}^\infty(\Omega,\mathbb R^p)$),
any solution of the system $v_x=Av+F$ in ${\rm C}^0_t(\mathcal D'_x)(\Omega,\mathbb R^p)$ (resp.\ ${\rm C}^\infty_t(\mathcal D'_x)(\Omega,\mathbb R^p)$)
is a regular distribution associated with a classical (resp.\ smooth) solution  of this system.
\end{proposition}

\begin{proof}
In the notation of the proof of Proposition~\ref{pro:DistrSolutionsOfLinSystemsOfODEsWithParameter},
if $v$ is ${\rm C}^0$-semiregular in~$t$ (resp.\ ${\rm C}^\infty$-semi\-regular in~$t$),
then $\tilde v$ is of the same semiregularity in~$t$
and hence the corresponding tuple $\zeta$ belongs to ${\rm C}(\pr_t\Omega,\mathbb R^p)$
(resp.\ ${\rm C}^\infty(\pr_t\Omega,\mathbb R^p)$).
 \end{proof}

\begin{appendices}

\section{Distributions with vanishing partial derivatives}\label{sec:DistributionsWithVanishingPartialDerivatives}

In this appendix we collect some results
required in Sections~\ref{sec:DistrSolutionsOfLinODEsWithParameter} and~\ref{sec:LinSystemsOfODEsWithParameter}
for deriving the general form of distributional solutions to linear (systems of) ODEs.

For a distribution $u\in \mathcal D'(\mathbb R^{n+1})$ it is well known (cf.\ \cite[Chapitre~IV, \S~5]{Schwartz1966}, \cite[Theorem~4.3.4]{Friedlander1998})
that $\p_{x_{n+1}}u=0$ if and only if $u$ is of the form $v\otimes T_{\mathbf 1_{\mathbb R}}$ for some $v\in\mathcal D'(\mathbb R^n)$.
For $u\in\mathcal D'(\Omega)$ with~$\Omega$ an arbitrary open subset of $\mathbb R^n$ such a result cannot be expected.
Nevertheless, we shall show that if $\Omega$ is $x_{n+1}$-simple, then a suitable generalization does indeed hold.

We first note that \cite[Theorem 4.3.4]{Friedlander1998} remains true for more general products,
and we include a proof for the sake of completeness:

\begin{theorem}\label{thm:distributional_derivative_tensor_product}
Let $X\subset \mathbb R^n$ be open, $Y=(a,b)$, $-\infty\le a <b\le \infty$, $n\in \mathbb N_0$
(setting $X\times Y:=Y$ in case $n=0$),
and let $u\in \mathcal D'(X\times Y)$. Then
\[
\p_y u = 0 \ \Leftrightarrow \exists\ v\in \mathcal D'(X)\colon\ u(x,y) = v\otimes T_{\mathbf 1_Y}.
\]
\end{theorem}
\begin{proof}
$(\Leftarrow)$: $\p_y(v\otimes T_{\mathbf 1_Y}) = v\otimes\p_y T_{\mathbf 1_Y} = v\otimes 0 = 0$.

\noindent$(\Rightarrow)$: Pick some $\chi\in \mathcal D(Y)$ with $\int_Y \chi(y)\,{\rm d}y=1$
and define $v\colon \mathcal D(X)\to \mathbb R$ by
\[
\lara{v}{\psi}:=\lara{u(x,y)}{\psi(x)\otimes\chi(y)}\quad \mbox{for all}\quad \psi\in \mathcal D(X).
\]
Then $v\in \mathcal D'(X)$ and for any $\varphi\in \mathcal D(X\times Y)$ we have
\begin{align*}
\lara{v\otimes T_{\mathbf 1_Y}}{\varphi}
&= \blara{v(x)}{\lara{T_{\mathbf 1_Y}}{\varphi(x,y)}}=\blara{v(x)}{\textstyle\int_Y\varphi(x,y)\,{\rm d}y}\\[.5ex]
&= \blara{u(x,y)}{\textstyle\int_Y\varphi(x,y')\,{\rm d}y'\otimes \chi(y)}.
\end{align*}
Hence $\lara{u-v\otimes T_{\mathbf 1_Y}}{\varphi} = \lara{u}{\phi}$, where $\phi(x,y)=\varphi(x,y)-\int_Y\varphi(x,y')\,{\rm d}y'\otimes \chi(y)$.
Now for any $x\in X$ we have $\int_Y \phi(x,y)\,{\rm d}y = 0$, so that $\psi(x,y):=\int_a^y \phi(x,y')\,{\rm d}y'$ defines an
element of $\mathcal D(X\times Y)$ and satisfies $\p_y \psi=\phi$. Consequently,
\[
\lara{u-v\otimes T_{\mathbf 1_Y}}{\varphi} = \lara{u}{\phi} = \lara{u}{\p_y \psi} = -\lara{\p_y u}{\psi} = 0,
\]
which completes the proof.
\end{proof}

In analogy to Definition \ref{def:XSimpleSet} we say that an open subset $\Omega$ of $\mathbb R^n_x\times \mathbb R_y$
(where the subscripts refer to the names of the corresponding variables)
is \emph{$y$-simple} if, for all $x\in \mathbb R^n$,
the intersection $\Omega_x:=(\{x\}\times \mathbb R) \cap \Omega$ is connected
or is the empty set.
Using this terminology, we have:

\begin{theorem}\label{thm:distributional_derivative_tensor_product_y_simple}
Let $\Omega\subset \mathbb R^n_x\times \mathbb R_y$ be open and $y$-simple ($n\in \mathbb N_0$), and let $u\in \mathcal D'(\Omega)$. Then
\[
\p_y u = 0 \ \Leftrightarrow \exists\ v\in \mathcal D'(\pr_x(\Omega))\colon\ u(x,y) = (v\otimes T_{\mathbf 1_{\mathbb R_y}})\big|_\Omega\,.
\]
(Here $\pr_x$ denotes the projection into $\mathbb R^n_x$.)
\end{theorem}

\begin{proof} Only the direction $(\Rightarrow)$ requires a proof. Thus let $\{X_i\times Y_i\mid i\in J\}$ be an open
cover of~$\Omega$ with $X_i$ open in $\mathbb R^n_x$ and $Y_i$ open intervals in $\mathbb R_y$. Then $\p_y (u\big|_{X_i\times Y_i})=0$
for all $i\in J$, and so by Theorem \ref{thm:distributional_derivative_tensor_product} there exist
$v^i\in \mathcal D'(X_i)$ such that $u\big|_{X_i\times Y_i} = v^i\otimes T_{\mathbf 1_{Y_i}}$.

We now show that $\{v^i\mid i\in J\}$ forms a coherent family of distributions associated to the covering
$\{X_i\mid i\in J\}$ of $\pr_x\Omega$. To see this, suppose that $X_i\cap X_j\not=\varnothing$. We have to
show that then $v^i\big|_{X_i\cap X_j} = v^j\big|_{X_i\cap X_j}$. We distinguish two cases:

First, if $Y_i\cap Y_j\not=\varnothing$, then $U:=(X_i\times Y_i)\cap (X_j\times Y_j)\not=\varnothing$,
and therefore
\[
(v^i\otimes T_{\mathbf 1_{Y_i}})\big|_U = u\big|_U = (v^j\otimes T_{\mathbf 1_{Y_j}})\big|_U,
\]
so that indeed $v^i\big|_{X_i\cap X_j} = v^j\big|_{X_i\cap X_j}$.

Second, suppose that $Y_i\cap Y_j=\varnothing$ and fix arbitrary $x\in X_i\cap X_j$ and $y_i\in Y_i$, $y_j\in Y_j$, assuming without
loss of generality that $y_i<y_j$. Since $\Omega$ is $y$-simple, we may pick a finite subset~$J'$ of~$J$
such that $\{X_k\times Y_k\mid k\in J'\}$
is a minimal covering of $\{x\}\times [y_i,y_j]$. Let $J'=\{k_1,\dots,k_p\}$, $Y_{k_l}=(a_{k_l},b_{k_l})$, $l=1,\dots,p$,
and assume without loss of generality that $a_{k_1}<a_{k_2}<\dots <a_{k_p}$.
Then $V:=\bigcap_{k\in J'}X_k$ is an open neighborhood of $x$, and by the previous case we have
\[
v^i\big|_V = v^{k_1}\big|_V = \dots = v^{k_p}\big|_V = v^j\big|_V,
\]
which allows us to conclude that $v^i\big|_{X_i\cap X_j} = v^j\big|_{X_i\cap X_j}$ also in this case.

By the sheaf property of distributions \cite[Chapitre~I, \S~3, Th\'eor\`eme~IV]{Schwartz1966}, it follows that there
exists a unique $v\in \mathcal D'(\pr_x(\Omega))$ such that $v\big|_{X_i}=v^i$ for all $i\in J$. Thus
\[
u\big|_{X_i\times Y_i} = v^i\otimes T_{\mathbf 1_{Y_i}} = v\big|_{X_i} \otimes T_{\mathbf 1_{Y_i}}
= (v\otimes T_{\mathbf 1_{\mathbb R_y}})\big|_{X_i\times Y_i}.
\]
Again by the sheaf property of distributions, $u=(v\otimes T_{\mathbf 1_{\mathbb R_y}})\big|_\Omega$.
\end{proof}

\begin{remark}\looseness=-1
The condition on the simplicity of the domain with respect to the variable involved in the distributional derivative
is essential in Theorem~\ref{thm:distributional_derivative_tensor_product_y_simple}.
Indeed, given a nonempty open set $\Omega\subset \mathbb R^n_x\times \mathbb R_y$ ($n\in \mathbb N_0$) that is not $y$-simple,
we may take some $x_0\in\mathbb R^n_x$ such that the intersection $\Omega_{x_0}:=(\{x_0\}\times \mathbb R) \cap \Omega$ is non-empty
and not connected.
For $-\infty\leqslant y_0<y_1\leqslant y_2<y_3\leqslant+\infty$
such that $\{x_0\}\times(y_0,y_1)$ and $\{x_0\}\times(y_2,y_3)$ are connected components of~$\Omega_{x_0}$
and for any different and nonzero $c_1,c_2\in\mathbb R$, consider the distribution
\[
u:=\left(\delta_{x_0}\otimes\big(c_1T_{\mathbf 1_{(y_0,y_1)}}+c_2T_{\mathbf 1_{(y_2,y_3)}}\big)\right)\Big|_\Omega\in\mathcal D'(\Omega).
\]
Then $u$ is not of the form given in Theorem~\ref{thm:distributional_derivative_tensor_product_y_simple}
although $\p_xu=0$.
\end{remark}

\end{appendices}

\subsection*{Acknowledgements}

The authors thank Michael Grosser and Galyna Popovych for helpful discussions and interesting comments.
The research of ROP was supported by the Austrian Science Fund (FWF), projects P25064 and P30233.
ROP is also grateful to the project No.\ CZ.$02.2.69\/0.0/0.0/16\_027/0008521$
``Support of International Mobility of Researchers at SU'' which supports international cooperation. %\todo


\begin{thebibliography}{99}
\footnotesize
\itemsep=0ex

\bibitem{Amann1990}
Amann H.,
{\it Ordinary differential equations},
De Gruyter Studies in Mathematics, 13, Walter de Gruyter \& Co., Berlin, 1990.

\bibitem{Friedlander1998}
Friedlander F.G.,
{\it Introduction to the theory of distributions},
Second edition. With additional material by M. Joshi, Cambridge University Press, Cambridge, 1998.

\bibitem{Hartman2002}
Hartman P.,
{\it Ordinary differential equations},
Corrected reprint of the second (1982) edition [Birkh\"auser, Boston, MA],
Classics in Applied Mathematics, 38, Society for Industrial and Applied Mathematics (SIAM), Philadelphia, PA, 2002.

\bibitem{Hoermander1976}
H\"ormander L.,
{\it Linear partial differential operators},
Springer-Verlag, Berlin--New York, 1976.

\bibitem{Hoermander1983}
H\"ormander L.,
{\it The analysis of linear partial differential operators. I. Distribution theory and Fourier analysis},
Grundlehren der Mathematischen Wissenschaften, vol.~256,  Springer-Verlag, Berlin, 1983.

\bibitem{Hoermander1983_2}
H\"ormander L.,
{\it The analysis of linear partial differential operators. I. Distribution theory and Fourier analysis},
Grundlehren der Mathematischen Wissenschaften, vol.~257,  Springer-Verlag, Berlin, 1983.

\bibitem{MadsenTornehave1997}
Madsen I. and Tornehave J.,
{\it From calculus to cohomology. de Rham cohomology and characteristic classes},
Cambridge University Press, Cambridge, 1997.

\bibitem{Malgrange&Garding1961}
Malgrange B. and Garding L.,
Op\'erateurs diff\'erentiels partiellement hypoelliptiques et partiellement elliptiques,
{\it Math. Scand.} {\bf 9} (1961), 5--21. (French)

\bibitem{Marsden&Tromba}
Marsden J.E. and Tromba A.J.,
{\it Vector Calculus},
Fifth edition, Freeman and Company, New York, 2003.

\bibitem{Matveev1979}
Matveev V.B.,
Darboux transformation and explicit solutions of the Kadomtcev--Pet\-viasch\-vily equation, depending on functional parameters,
{\it Lett. Math. Phys.} {\bf 3} (1979), 213--216.

\bibitem{Matveev&Salle1991}
Matveev V.B. and Salle M.A.,
{\it Darboux transformations and solitons},
Springer-Verlag, Berlin, 1991.

\bibitem{Popovych&Kunzinger&Ivanova2008}
Popovych R.O., Kunzinger M. and Ivanova N.M.,
Conservation laws and potential symmetries of linear parabolic equations,
{\it Acta. Appl. Math.} {\bf 100} (2008), 113--185, arXiv:0706.0443.

\bibitem{Schwartz1957}
Schwartz L.,
Distributions semi-r\'eguli\`eres et changements de coordonn\'ees,
{\it J. Math. Pures Appl. (9)} {\bf 36} (1957), 109--127. (French)

\bibitem{Schwartz1966}
Schwartz L.,
{\it Th\'eorie des distributions},
Publications de l'Institut de Math\'ematique de l'Universit\'e de Strasbourg, Hermann, Paris, 1966. (French)

\bibitem{Walter1998}
Walter W.,
{\it Ordinary differential equations},
Graduate Texts in Mathematics, 182, Springer-Verlag, New York, 1998.


\end{thebibliography}
\end{document}